\documentclass[10pt, leqno]{article}

\usepackage{amsmath,amssymb,amsthm, epsfig}

\usepackage{hyperref}

\usepackage{amsmath}

\usepackage[pdftex]{color}

\usepackage{color}

\usepackage{ulem}

\usepackage{mathrsfs}

\usepackage{wasysym}

\usepackage{stackrel}

\title{\large{\bf Weighted Orlicz regularity for fully nonlinear elliptic equations with oblique derivative at the boundary via asymptotic operators}}
\author{\it by \smallskip \\
 Junior da S. Bessa\footnote{\noindent Universidade Federal do Cear\'{a}. Department of Mathematics. Fortaleza - CE, Brazil. \noindent \texttt{E-mail address: \url{junior.bessa@alu.ufc.br}}}}

\newlength{\hchng}
\newlength{\vchng}
\def \R {\mathbb{R}}

\def \loc {\mathrm{loc}}

\def \Leb {\mathcal{L}^n}

\newcommand{\defeq}{\mathrel{\mathop:}=}

\newtheorem{theorem}{Theorem}[section]

\newtheorem{lemma}[theorem]{Lemma}

\newtheorem{proposition}[theorem]{Proposition}

\theoremstyle{definition}

\newtheorem{definition}[theorem]{Definition}

\theoremstyle{remark}

\newtheorem{remark}[theorem]{Remark}

\numberwithin{equation}{section}

\newcommand{\intav}[1]{\mathchoice {\mathop{\vrule width 6pt height 3 pt depth  -2.5pt
\kern -8pt \intop}\nolimits_{\kern -6pt#1}} {\mathop{\vrule width
5pt height 3  pt depth -2.6pt \kern -6pt \intop}\nolimits_{#1}}
{\mathop{\vrule width 5pt height 3 pt depth -2.6pt \kern -6pt
\intop}\nolimits_{#1}} {\mathop{\vrule width 5pt height 3 pt depth
-2.6pt \kern -6pt \intop}\nolimits_{#1}}}

\begin{document}
\maketitle

\begin{abstract}

We prove weighted Orlicz-Sobolev regularity for fully nonlinear elliptic equations with oblique boundary condition under asymptotic conditions of the following problem
$$
\left\{
\begin{array}{rclcl}
 F(D^2u,Du,u,x) &=& f(x)& \mbox{in} &   \Omega \\
 \beta \cdot Du + \gamma u&=& g &\mbox{on}& \partial \Omega,
\end{array}
\right.
$$
where $\Omega$ is a bounded domain in $\mathbb{R}^{n}$($n\ge 2$), under suitable assumptions on the source term  $f$, data $\beta, \gamma$ and $g$. Our approach guarantees such estimates under conditions where the governing operator $F$ does not require a convex (or concave) structure. We also deal with weighted Orlicz-type estimates for the obstacle problem with oblique derivative condition on the boundary. As a final application, the developed methods provide weighted Orlicz-BMO regularity for the Hessian, provided that the source lies in that space and in weighted Orlicz space associated.

\vspace{0.2cm}

\end{abstract}

\section{Introduction}

\hspace{0.4 cm} We established regularity for fully nonlinear elliptic equations  with oblique boundary condition of the form
\begin{equation}\label{1.1}
\left\{
\begin{array}{rclcl}
 F(D^2u,Du,u,x) &=& f(x)& \mbox{in} &   \Omega \\
\mathfrak{B}(Du,u,x) &=& g(x) &\mbox{on}& \partial \Omega,
\end{array}
\right.
\end{equation}
where  $\Omega$ is a bounded domain in $\mathbb{R}^{n}$ for $n\ge 2$ with regular boundary, with the term source $f$ in the context of weighted Orlicz spaces (see Definition \ref{weightedorliczspaces}). In this case, we are dealing with a generalization of the $W^{2,p}$ interior estimates of fully nonlinear
elliptic equations that were obtained by Caffarelli in \cite{Caff1} in the context of problems with oblique boundary conditions. We assume that  $F:\textit{Sym}(n)\times \R^n \times \R \times \Omega \to \R$  is a uniformly elliptic operator with ellipticity constants $\lambda$ and $\Lambda$, i.e., there are constants $0 < \lambda \le \Lambda < \infty$ such that 
\begin{equation}\label{Unif.Ellip.}
\lambda\|\mathrm{N}\|\leq F(\mathrm{M}+\mathrm{N},\varsigma , r, x)-F(\mathrm{M}, \varsigma, r, x) \leq \Lambda \|\mathrm{N}\|
\end{equation}
 for every matrices $\mathrm{M}, \mathrm{N} \in \textit{Sym}(n)$ with $\mathrm{N} \ge 0$ (in the usual partial ordering on symmetric matrices) and points $(\varsigma, r, x)\in \mathbb{R}^n \times \R \times \Omega $. Equation \eqref{1.1} has a great interest of study by researchers in mathematics and other areas, we can mention to \cite{Bessa}, \cite{BJ}, \cite{BJ1},  \cite{daSR19} and \cite{PT} works that studied this type of problem along the lines of motivation of Caffarelli's famous work \cite{Caff1}. 
 The operator $\mathfrak{B}$ governs the oblique boundary condition, more precisely we have that $\mathfrak{B}$ is prescribed in terms of directional derivative with respect to a vector field $\beta(x) : \partial \Omega \rightarrow \mathbb{R}^n$, defined on $\partial \Omega$. In general, an oblique boundary condition is given by the following form:
\begin{equation} \label{Obl}
	\mathfrak{B}(Du, u,x) := \beta(x) \cdot Du + \gamma(x)u,
\end{equation}
where  $g$, $\beta$ and $\gamma$ are functions.   In the theory of the tangential oblique derivative problem, the word oblique means that $|\beta(x) \cdot \nu(x)| \ge \mu_0 >0$ on $\partial \Omega$, where $\nu(x)$ denotes the inner normal of $\Omega$.  This tells us that $\beta$ is not in a direction orthogonal to $\nu$ in $\partial\Omega$ and thus  $\beta$ is in a "sloping direction" with respect to the normal $\nu$. This imposition is due to the complementary condition of Shapiro–Lopatinskii, which states that the problem \eqref{1.1} is well posed (i.e., regular, non-degenerate, non-singular) if, and only if, $\beta(x)\cdot \nu(x)$  is non-null in $\partial\Omega$ (For more details, cf. \cite{MaPaVi}).  This boundary condition appears, in fact, as the interpretation of the standard capillary condition in the level set approach to motion of hypersurfaces by curvature (when $\beta$ makes an angle more than some fixed level on $\partial \Omega$) (see for instance \cite{Sato} and \cite{Giga}). The previous comment motivates our considerations in the oblique condition, we always assume that
$$
\beta(x) \cdot \nu(x) \ge \mu_0 \quad \textrm{on} \quad \partial \Omega \quad \textrm{and} \quad \|\beta\|_{L^{\infty}(\partial \Omega)} \le 1
$$
for some positive constant $\mu_0$.  The most familiar case in the mathematical world is the Neumann boundary condition, that is, the case $\beta = \nu(x)$ and $\gamma =0$ where $\nu(x)$ is the outward normal vector of $\partial \Omega$. The Robin boundary condition, i.e., when $\beta = \nu(x)$ and $\gamma \not= 0$ is also an example of an oblique boundary condition.

The justification for so many investigations regarding the oblique condition problem is due to its appearance in most general contexts than pure mathematics, as for instance Mathematical Physics and Applied Mathematics. An illustration of this is given in the theory of Markov processes (e.g., Brownian motion). The oblique operator $\mathfrak{B}$ governs such a problem where the equation that describes such a phenomenon is $\mathfrak{B}(Du,u,x)=g$ being the following interpretation: the term $\beta\cdot Du$ describes the reflection of the process along the $\beta$ field, while the second term $\gamma u$ is related to the absorption phenomenon. We also mention applications in celestial body mechanics, stochastic control theory, capillary free surfaces, reflected shocks in transonic flows, and so on. We recommend Lieberman's book \cite{Lieberman} to readers for a deeper understanding of these concepts and applications, mainly in the area of PDE's.

The $W^{2,p}$ regularity theory for nonlinear elliptic equations is a great research attraction until the present day. At the end of the 1980s, Caffarelli in his famous work \cite{Caff1} established the bases for the theory $W^{2,p}$ for
nonlinear elliptic equations for $p > n$ using the perturbation method. Escauriaza extended such estimates to $n-\varepsilon < p$ in \cite{Es93} in the year 1993, where $\varepsilon\in (0, n/2)$ can
be teaced back to \cite{Coi}. In 1997, \'Swiech provided $W^{2,p}$ estimates for $L^{p}$-viscosity solutions by  $p \geq n-\varepsilon$ (see \cite[Theorem 3.1]{Sw}). Recently, Winter \cite{Winter} obtained global estimates analogous to interior estimates in \cite{Sw}. It should be noted that all roles mentioned above depend on $C^{1,1}$ estimates for equations with constant coefficients. A priori $W^{2,p}$ estimates when it does not impose any extra assumption of regularity in the nonlinearity of $F$ was obtained by Pimentel and Teixeira in \cite{PT} for asymptotically convex models. Recently, da Silva and Ricarte in \cite{daSR19} generalized Winter’s work \cite{Winter} with an asymptotic approach. We can also highlight the work of Byun et al. in \cite{BLOK} in the context of weighted Orlicz spaces among many other works.

The study of this theory for fully nonlinear equations with oblique boundary condition improves many recent works of great relevance. For example, we can cite the work \cite{BJ1} of Byun et al. on estimates $W^{2,p}$ for solutions of $\eqref{1.1}$ when $F$ is convex and $\gamma\equiv g\equiv 0$. As early as 2021, Zhang et al. in \cite{ZZZ21} generalized the work \cite{BJ1} in an asymptotic context. A little later in 2022, Bessa et al. in \cite{Bessa}, developed a very general version of estimates of the type $W^{2,p}$ for the \eqref{1.1} problem in an asymptotic context without requiring any extra regularity assumption, e.g. concavity, convexity or smoothness hypothesis. Recently Bessa and Ricarte \cite{BessaRicarte} in 2023, treated in this asymptotic line weighted Lorentz-type estimates for the \eqref{1.1} problem. 

Motivated by this, our purpose is to obtain estimates for the problem \eqref{1.1} in the context of weighted Orlicz spaces with asymptotic conditions, paralleling the works \cite{Bessa}, \cite{BessaRicarte} and \cite {Lee}, and to explore some applications in different contexts, namely in the obstacle problem and in weighted Orlicz-BMO estimates. To achieve the desired, we will study the problem \eqref{1.1} under asymptotic conditions on the governing operator $F$. More precisely, our estimative depends of behavior of the fully nonlinear operator $F$ near infinity  in $Sym(n)$. The idea of an asymptotic operator is well put when we introduce the notion of \textit{Recession operator}, whose terminology comes from Giga-Sato's work \cite{GS01} on the theory of Hamilton-Jacobi PDEs.

\begin{definition}\label{DefAC}
For an operator $F$ as in \eqref{1.1}, we define the associated \textbf{recession operator} and use the notation $F^{\sharp}$ putting
\begin{equation}\label{Reces}
  \displaystyle F^{\sharp}(\mathrm{X}, \varsigma, s, x) \defeq  \lim_{\tau \to 0^{+}} \tau \cdot F\left(\frac{1}{\tau}\mathrm{X}, \varsigma, s, x\right),
\end{equation}
 for all $\mathrm{X} \in \textrm{Sym}(n)$, $\varsigma \in \mathbb{R}^n$, $s \in \mathbb{R}$ and $x \in \Omega$.
\end{definition}

In this context, we say that $F$ is an asymptotically regular fully nonlinear elliptic operator if $F^{\sharp}$ exists and is a fully nonlinear elliptic operator. Then, the notion of asymptotically regular problem goes back to chipot-evans work \cite{Ch} treating about certain problems in calculating variations. For more details about the recession operator and issues involving the asymptotic condition, we recommend \cite{Bessa}, \cite{daSR19}, \cite{PT}, \cite{ZZZ21} and \cite{BOW16}.


The approximation via asymptotic operators gives us a great differential in the ideas of the present work since, instead of imposing structural hypothesis on the $F$ governing operator, we will make assumptions on the recession operator $F^{\sharp}$ (for example, convexity/concavity or \textit{ suitable a priori} estimates). In this line of reasoning, through a geometric tangential analysis, we can obtain such good regularity properties, so that the solutions associated with the original operator inherit this regularity through compactness and stability processes. With this idea, under certain hypothesis, a theory of regularity $W^{2,p}$ is possible and, therefore, linked to the context of weighted Orlicz regularity.

It is worth to mention that problems in the same line of research of the present work have been gaining prominence, see \cite{Bessa}, \cite{BessaRicarte}, \cite{BJ}, \cite{BJ1}, \cite{BLOK}, \cite{MF}, \cite{Kry13},\cite{Kry17}, \cite{PT}, \cite{JP}, \cite{ST} and \cite{ZhangZhenglorentz} for an incomplete list of contributions.

\subsection{Main Results}
\hspace{0.4cm} Here we will set up some notations and appreciate some necessary results to enunciate the main results and for their use in the rest of the manuscript.

In the present text, for a given $r >0$, we denote by $\mathrm{B}_r(x)$ the ball of radius $r$ centered at $x=(x^{\prime},x_n) \in \R^n$, where $x^{\prime}=(x_1,x_2,\ldots, x_{n-1}) \in \R^{n-1}$. We write $\mathrm{B}_r = \mathrm{B}_r(0)$ and $\mathrm{B}^+_r = \mathrm{B}_r \cap \mathbb{R}^n_+$. Also, we write $\mathrm{T}_r \defeq \{(x^{\prime},0) \in \mathbb{R}^{n-1} : |x^{\prime}| < r\}$ and $\mathrm{T}_r(x_0) \defeq \mathrm{T}_r + x^{\prime}_0$, where $x^{\prime}_0 \in \mathbb{R}^{n-1}$. Finally, $\textrm{Sym}(n)$ will denote the set of all $n \times n$ real symmetric matrices.

We denote by
$$
	\mathcal{Q}^n_r(x_0) \colon= \left( x_0-\frac{r}{2}, x_0 + \frac{r}{2}\right)^d
$$
the $d$-dimension cube of side-length $r$ and center $x_0$. In the case $x_0=0$, we will write $\mathcal{Q}^d_r$.  Furthermore, if $d=n$, we will write $\mathcal{Q}_r(x_0)$.

Recall that a function $\Phi:[0,+\infty)\to [0,+\infty)$ is said to be an $N$-function if it is convex, increasing and continuous, and satisfies that $\Phi(t)>0, \ \forall t>0$, $\Phi(0)=0$,
\begin{eqnarray*}
\lim_{t \to 0^{+}}\frac{\Phi(t)}{t}=0 \  \mbox{and} \ \lim_{t\to+\infty}\frac{\Phi(t)}{t}=+\infty.
\end{eqnarray*}
In this context, we say that an $N$-function $\Phi$ satisfies the $\Delta_{2}$-condition (resp. $\nabla_{2}$-condition) if there is a constant $\mathrm{K}>1$ (resp. $\mathrm{L}>1$ ) such that
\begin{eqnarray*}
\Phi(2t)\leq \mathrm{K}\Phi(t) \ \left(\mbox{resp.} \ \Phi(t)\leq \frac{1}{2\mathrm{L}}\Phi(\mathrm{L}t)\right), \ \forall t>0.
\end{eqnarray*}
We will use the notation $\Phi\in\Delta_{2}$ (resp. $\Phi\in \nabla_{2}$) to indicate that $\Phi$ complies with $\Delta_{2}$-condition (resp. $\nabla_{2}$-condition). Furthermore, the notation $\Phi\in\Delta_{2}\cap \nabla_{2}$ will be used to indicate that $\Phi\in\Delta_{2}$ and $\Phi\in\nabla_{2}$.
For a function $\Phi\in\Delta_{2}\cap\nabla_{2}$, we can define its lower index by
\begin{eqnarray*}
i(\Phi)=\lim_{t \to 0^{+}}\frac{\log(h_{\Phi}(t))}{\log t}=\sup_{0<t<1}\frac{\log(h_{\Phi}(t))}{\log t},
\end{eqnarray*}
where
\begin{eqnarray*}
h_{\Phi}(t)=\sup_{s>0}\frac{\Phi(st)}{\Phi(s)}, \ t>0.
\end{eqnarray*}
Illustrating the above definitions, we have that the functions $\Phi(t)=t^{q}$ and $\overline{\Phi}(t)=t^{q}\log(t+1)$, for $q>1$, are $N$-functions that satisfy the condition $\Delta_{2}\cap\nabla_{2}$ with $i(\Phi)=i(\overline{\Phi})=q>1$.

On the other hand, we recall the definition of \textit{weights}. We say that  $\omega$ is a \textit{weight}, if is a locally integrable nonnegative function  that takes values in $(0,+\infty)$ almost everywhere. This case, we identify $\omega$ with the measure 
\begin{eqnarray*}
\omega(E)=\int_{E}\omega(x)dx,
\end{eqnarray*} 
for all Lebesgue measurable set $E\subset \mathbb{R}^{n}$. We say that be a weight $\omega$ belong to \textit{Muckenhoupt class} $A_{q}$ for $q\in(1,\infty)$, denoted by $\omega\in A_{q}$, if 
\begin{eqnarray*}
[\omega]_{q}\defeq \sup_{B\subset \mathbb{R}^{n}}\left(\intav{B}\omega(x)dx\right)\left(\intav{B}\omega(x)^{\frac{-1}{q-1}}dx\right)^{q-1}<\infty,
\end{eqnarray*}
where the supremum is taken over all balls $B\subset \mathbb{R}^{n}$.

The concept of $A_{p}$ classes was introduced by Muckenhoupt in the mid-1970s in \cite{Muckenhoupt} and is of central importance in modern harmonic analysis and its applications.

With the above definitions, we can talk about the main class of functions that we will work with throughout this article. 

\begin{definition}\label{weightedorliczspaces}
The \textit{weight Orlicz space} $L^{\Phi}_{\omega}(E)$ for be an $N$-function $\Phi\in \Delta_{2}\cap\nabla_{2}$, Lebesgue measurable set $E\subset \mathbb{R}^{n}$ and weighted $\omega$ is the set of all measurable functions $g$ in $E$ such that
\begin{eqnarray*}
\rho_{\Phi,\omega}(g)=:\int_{E}\Phi(|g(x)|)\omega(x)dx<+\infty.
\end{eqnarray*}
Which by the condition $\Phi\in\Delta_{2}\cap\nabla{2}$, is a reflexive Banach space when endowed with the following Luxemburg norm
\begin{eqnarray*}
\|g\|_{L^{\Phi}_{\omega}(E)}=:\inf\left\{t>0:\rho_{\Phi,\omega}\left(\frac{g}{t}\right)\leq 1\right\}.
\end{eqnarray*}
Furthermore, the \textit{weighted Orlicz-Sobolev space} $W^{k,\Phi}_{\omega}(E)$ (for an integer $k\geq0$) is the set of all measurable functions $g$ in $E$ such that all distributional derivates $D^{\alpha}g$, for multiindex $\alpha$ with length $|\alpha|=0,1,\cdots,k$ also belong to $L^{\Phi}_{\omega}(E)$ with norm is given by
\begin{eqnarray*}
\|g\|_{W^{k,\Phi}_{\omega}(E)}\defeq \sum_{j=1 }^{k}\|D^{j}g\|_{L^{\Phi}_{\omega}(E)}.
\end{eqnarray*} 
\end{definition}
It is noteworthy that weighted Orlicz spaces appear naturally in harmonic analysis, such as in the study of potentials and Riesz transforms(cf. \cite{KokiMiro} for more details).
\begin{remark}
We point out here that if $\Phi(t)=t^{q}$ for $q>1$, $L^{\Phi}_{\omega}(E)$ coincides with the classic Lebesgue space with exponent variable $L^{q}_{\omega}(E)$ as well as $W^{k,\Phi}_{\omega}(E)$ coincides with the weighted Sobolev space  $W^{k,q}_{\omega}(E)$. In particular, if $\omega=1$ it follows that $L^{\Phi}_{\omega}(E)$ and $W^{k,\Phi}_{\omega}(E)$ are the spaces from Lebesgue and Sobolev $L^{q}(E)$ and $W^{k,q}(E)$, respectively.
\end{remark}

We observe that for an $N$-function $\Phi\in\Delta_{2}\cap\nabla_{2}$, there exist two constants $p_{1}$ and $p_{2}$ such that $1<p_{1}\leq p_{2}<+\infty$
\begin{eqnarray}\label{propriedadedei(Phi)}
c^{-1}\min\{s^{p_{1}},s^{p_{2}}\}\Phi(t)\leq\Phi(st)\leq c\max\{s^{p_{1}},s^{p_{2}}\}\Phi(t), \ \forall t,s\geq 0,
\end{eqnarray}
where $c>0$ is idependent of $t$ and $s$(cf. \cite{KokiMiro}).  Thus, by the above definitions follows the following inclusions
\begin{eqnarray*}
L^{\infty}(E)\subset L^{p_{2}}_{\omega}(E)\subset L^{\Phi}_{\omega}(E)\subset L^{p_{1}}_{\omega}(E)\subset L^{1}(E).
\end{eqnarray*}
In addition to this fact, let us observe that $i(\Phi)>1$, since it is given as the supremum of all constants $p_{1}$ satisfying \eqref{propriedadedei(Phi)} for all $s\geq1$(cf. \cite{fioreza} for more details). Therefore, it makes sense to speak of Muckenhoupt's weight class $A_{i(\Phi)}$.

It is worth noting that under these conditions we still have to
\begin{eqnarray}\label{estimativadamodular}
\rho_{\Phi,\omega}(g)\leq \mathrm{C}(\|g\|_{L^{\Phi}_{\omega}(E)}^{p_{2}}+1),
\end{eqnarray}
where $\mathrm{C}>0$ is independent of $g$(cf. \cite{BLOK}).

Finally, with these definitions and facts above, we observe that we can continuously embed $L^{\Phi}_{\omega}(\Omega)$ over the some Lebesgue space via the following result proved in \cite[Lemma 5]{BLOK}
\begin{lemma}\label{mergulhoorliczlebesgue}
Let $\Phi$ be an $N$-function such that $\Phi\in\Delta_{2}\cap\nabla_{2}$, $\omega\in A_{i(\Phi)}$ and $\Omega\subset \mathbb{R}^{n}$ be a bounded. Then there exists $p_{0}$ depends only to $i(\Phi)$ and $\omega$ with $1<p_{0}<i(\Phi)$ such that $L^{\Phi}_{\omega}(\Omega)$ is continuously embedded in $L^{p_{0}}(\Omega)$. Furthermore, the following estimate is valid
\begin{eqnarray*}
\|g\|_{L^{p_{0}}(\Omega)}\leq \mathrm{C}'\|g\|_{L^{\Phi}_{\omega}(\Omega)}, \ \forall g\in L^{\Phi}(\Omega),
\end{eqnarray*}
where $\mathrm{C}'=\mathrm{C}'(n,i(\Phi),\omega)>0$ is independent of $g$.
\end{lemma}
Throughout this manuscript, we will assume the following structural conditions on the problem $\eqref{1.1}$:

\begin{enumerate}
\item[(H1)](\textbf{Operator structure}) We assume that $F \in C^0(\text{Sym}(n), \R^n, \R, \Omega)$. Moreover, there are constants $0 < \lambda \le \Lambda$, $\sigma\geq 0$ and $\xi\geq 0$ such that
\begin{eqnarray}\label{EqUnEll}
\mathcal{M}^{-}_{\lambda,\Lambda}(\mathrm{X}-\mathrm{Y}) - \sigma |\zeta_1-\zeta_2| -\xi|r_1-r_2| &\le& F(\mathrm{X}, \zeta_1,r_1,x)-F(\mathrm{Y},\zeta_2,r_2,x) \nonumber \\
&\le& \mathcal{M}^{+}_{\lambda, \Lambda}(\mathrm{X}-\mathrm{Y})+ \sigma |\zeta_1-\zeta_2| + \xi|r_1-r_2| \label{5}
\end{eqnarray}
for all $\mathrm{X},\mathrm{Y} \in \textit{Sym}(n)$, $\zeta_1,\zeta_2 \in \mathbb{R}^n$, $r_1,r_2 \in \mathbb{R}$, $x \in \Omega$, where
\begin{equation}
\mathcal{M}^{+}_{\lambda,\Lambda}(\mathrm{X}) \defeq  \Lambda \sum_{e_i >0} e_i +\lambda \sum_{e_i <0} e_i \quad \text{and} \quad \mathcal{M}^{-}_{\lambda,\Lambda}(\mathrm{X}) \defeq \Lambda \sum_{e_i <0} e_i + \lambda \sum_{e_i >  0} e_i,
 \end{equation}
 are the \textit{Pucci's extremal operators} and $e_i = e_i(\mathrm{X})$ ($1\leq i\leq n$) denote the eigenvalues of $\mathrm{X}$. By normalization question we assume that $F(0,0,0,x)=0$ for all $x\in \Omega$. From now on, an operator fulfilling $\text{(H1)}$ will be called $(\lambda, \Lambda, \sigma, \xi)-$elliptic operator.

\item[(H2)](\textbf{Hypothesis on the data}) The data satisfy $|f|^{n} \in L^{\Phi}_{\omega}(\Omega)$ for $\Phi\in\Delta_{2}\cap \nabla_{2}$ and $\omega\in A_{i(\Phi)}$, $g, \gamma \in C^{1,\alpha}(\partial \Omega)$ with $\gamma \le 0$ and $\beta \in C^{1,\alpha}( \partial \Omega; \mathbb{R}^n)$ for some $\alpha\in (0,1)$.

\item[(H3)](\textbf{Condition on the coefficients}) Fixed $x_0\in \Omega$, we define the quantity:
$$
\Psi_{F}(x; x_0) \defeq \sup_{\mathrm{X} \in \textrm{Sym}(n)} \frac{|F(\mathrm{X},0,0,x) - F(\mathrm{X},0,0,x_0)|}{1+\|\mathrm{X}\|},
$$
which measures the oscillation of the coefficients of $F$ around $x_0$. Furthermore, we abreviate following notation  $\psi_{F}(x; 0) = \psi_F(x)$.  Finally, we assume that the oscilation function  $\Psi_{F^{\sharp}}$ is assumed to be H\"{o}lder continuous function (in the $L^{p}$-average sense) for every $\mathrm{X} \in \textrm{Sym}(n)$. This means that, there exist universal constants(that is, if it depends only on $n, \lambda, \Lambda, p, \mu_0, \|\gamma\|_{C^{1,\alpha}(\partial \Omega)}$ and $\|\beta\|_{C^{1,\alpha}(\partial \Omega)}$) $\hat{\alpha} \in (0,1)$, $\theta_0 >0$ and $0 < r_0 \le 1$ such that
$$
\left( \intav{\mathrm{B}_r(x_0) \cap \Omega} \psi_{F^{\sharp}}(x,x_0)^{p} dx \right)^{1/p} \le \theta_0 r^{\hat{\alpha}}
$$
for $x_0 \in \overline{\Omega}$ and $0 < r \le r_0$.

\item[(H4)](\textbf{ $C^{1,1}$-interior estimates}) We assume that the recession operator $F^{\sharp}$ associated with the operator $F$ exists and has a priori $C^{1,1}_{loc}$ estimates. That means that $F^{\sharp}(D^{2}h) = 0$ in $\mathrm{B}_{1}$ in the viscosity sense, implies $h\in C^{1,1}(\overline{\mathrm{B}_{\frac{1}{2}}})$ and
\begin{eqnarray*}
\|h\|_{C^{1,1}(\overline{\mathrm{B}_{\frac{1}{2}}})}\leq c_{1}\|h\|_{L^{\infty}(\mathrm{B}_{1})}
\end{eqnarray*}
for a constant $c_{1}\geq1$.

\item[(H5)] (\textbf{$C^{1,1}$-local estimates for the oblique problem}) We shall assume that the recession operator $F^{\sharp}$ there exists and fulfils a up-to-the boundary $C^{1,1}$ \textit{a priori} estimates, i.e., for $x_0 \in \mathrm{B}^{+}_{1}$ and  $g_0 \in C^{1,\alpha}(\overline{\mathrm{T}_1})$ for some $\alpha \in (0, 1)$, there exists a solution $\mathfrak{h} \in C^{1,1}(\mathrm{B}^+_1) \cap C^0(\overline{\mathrm{B}^+_1})$ of
$$
\left\{
\begin{array}{rclcl}
 F^{\sharp}(D^2 \mathfrak{h},x_0) &=& 0& \mbox{in} &   \mathrm{B}^+_1 \\
\mathfrak{B}(D\mathfrak{h},\mathfrak{h},x)&=& g_{0}(x) &\mbox{on}& \mathrm{T}_1
\end{array}
\right.
$$
such that
$$
\|\mathfrak{h}\|_{C^{1,1}\left(\overline{\mathrm{B}^{+}_{\frac{1}{2}}}\right)} \le c_{2} \left(\|\mathfrak{h}\|_{L^{\infty}(\mathrm{B}^{+}_1)}+\|g_0\|_{C^{1,\alpha}(\overline{\mathrm{T}_{1}})}\right)
$$
for some constant $c_{2}>0$.
\end{enumerate}

 We now state our first main result, which endures estimates in context of weighted Orlicz spaces for viscosity solutions to asymptotically fully nonlinear equations.
\vspace{0.4cm}

\begin{theorem}[{\bf Weighted Orlicz estimates for oblique  boundary problems}]\label{T1}
Consider  $\Omega\subset \mathbb{R}^{n}$ $(n\geq 2)$ be a bounded domain with $\partial \Omega \in C^{2,\alpha}$ $(\alpha\in(0,1))$. Assume that structural assumptions (H1)-(H5) are in force and $u$ be an $L^{p}-$viscosity solution of \eqref{1.1} where $p=p_{0}n$ for the constant $p_{0}>1$ of Lemma \ref{mergulhoorliczlebesgue}. Then, $u \in W^{2,\Upsilon}_{\omega}(\Omega)$ where $\Upsilon(t)=\Phi(t^{n})$, with the estimate 
\begin{equation*}
\|u\|_{W^{2,\Upsilon}_{\omega}(\Omega)} \le \mathrm{C}\cdot\left(\|u\|^{n}_{L^{\infty}(\Omega)}+\|f\|_{L^{\Upsilon}_{\omega}(\Omega)}+\| g\|_{C^{1,\alpha}(\partial \Omega)}   \right),
\end{equation*}
where $\mathrm{C}$ is a positive constant that depends only on $n$,  $\lambda$,  $\Lambda$, $\xi$, $\sigma$, $p_{0}$, $p_{2}$, $\Phi$, $\omega$, $\mu_{0}$, $\|\beta\|_{C^{1,\alpha}(\partial\Omega)}$, $\|\gamma\|_{C^{1,\alpha}(\partial\Omega)}$ and $\|\partial \Omega\|_{C^{2,\alpha}}$.
\end{theorem}

We emphasize that our proof is based on the method employed in \cite{Bessa} for oblique derivate in boundary problems together with the ideas in the context of estimates in weighted Orlicz spaces that we can appreciate in \cite{BLOK} and \cite{Lee}. It is important to point out that such estimates obtained come with great aggregation to the theory of regularity, since only weighted Orlicz estimates are known for viscosity solutions of fully nonlinear equations in the interior context and with Dirichlet condition as the two cited references above. Another reason that corroborates the importance described above of Theorem \ref{T1} is that the governing operator $F$ has a weaker structure than convexity, because as in \cite{Bessa} we see examples of operators without the property of convexity\ concavity that they admit good estimates for the associated asymptotic operator.

On the other hand, with the validity of Theorem \ref{T1}, we will study obstacle problems with oblique boundary condition.  More precisely, we will show that under certain conditions solutions of this problem belong to the weighted Orlicz-Sobolev space. The obstacle problem that we will study is the following
\begin{equation} \label{obss1}
	\left\{
	\begin{array}{rclcl}
		F(D^2 u,Du,u,x) &\le& f(x)& \mbox{in} &   \Omega \\
		(F(D^2 u, Du, u,x) - f)(u-\psi) &=& 0 &\mbox{in}& \Omega\\
		u(x) &\ge& \psi(x) &\mbox{in}& \Omega\\
		\mathfrak{B}(Du,u,x) &=& g(x) &\mbox{on}& \partial \Omega,\\
	\end{array}
	\right.
\end{equation}
where $\psi\in W^{2,\Upsilon}_{\omega}(\Omega)$($\Upsilon$ as in Theorem \ref{T1}) is a given obstacle satisfying $\beta(x)\cdot D\psi+\gamma\psi \geq g$ a.e. on $\partial \Omega$. This problem has already been studied in \cite{BJ1}, $W^{2,p}$ estimates  were obtained for the problem $\eqref{obss1}$ when
$\gamma\equiv g=0$, $F$ convex and $\psi\in W^{2,p}(\Omega)$. Generalizing this work, \cite{Bessa} showed  $W^{2,p}$ estimates for the same problem \eqref{obss1} under asymptotic conditions for the operator $F$ and  \cite{BessaRicarte} studied in context of weighted Lorentz spaces. With this motivation, we will investigate along this line the regularity of the \eqref{obss1} problem in the reality of weighted Orlicz spaces.

In what follows, we will assume the following conditions:
\begin{enumerate}
	\item[(O1)] There exists a modulus of continuity $\eta: [0,+\infty) \to [0,+\infty)$ with $\eta(0)=0$, such that
	$$
	F(\mathrm{X}_1, \zeta, r,x_1) - F(\mathrm{X}_2,\zeta,r,x_2) \le \eta\left(|x_1-x_2|\right)\left[(|q| +1) + \alpha_0 |x_1-x_2|^2\right]
	$$
	for any $x_1,x_2 \in \Omega$, $\zeta \in \mathbb{R}^n$, $r \in \mathbb{R}$, $\alpha_0 >0$ and $\mathrm{X}_1,\mathrm{X}_2 \in \textrm{Sym}(n)$ satisfying
	$$
	- 3 \alpha_0
	\begin{pmatrix}
		\mathrm{Id}_n& 0 \\
		0& \mathrm{Id}_n
	\end{pmatrix}
	\leq
	\begin{pmatrix}
		\mathrm{X}_2&0\\
		0&-\mathrm{X}_1
	\end{pmatrix}
	\leq
	3 \alpha_0
	\begin{pmatrix}
		\mathrm{Id}_n & -\mathrm{Id}_n \\
		-\mathrm{Id}_n& \mathrm{Id}_n
	\end{pmatrix},	
	$$
	where $\mathrm{Id}_n$ is the identity matrix.
	\item[(O2)] $F$ is a \textit{proper operator}, i.e., $F$  fulfills the following condition,
	$$
	d\cdot (r_{2}-r_{1}) \leq F(\mathrm{X},\zeta,r_{1},x)-F(\mathrm{X},\zeta,r_{2},x),
	$$
	for any $\mathrm{X} \in \text{Sym}(n)$, $r_1,r_2 \in \mathbb{R}$, with $r_{1}\leq r_{2}$, $x \in \Omega$, $\zeta \in \mathbb{R}^n$, and some $d >0$.
\end{enumerate}

The conditions $\mbox{(O1)}$ and $\mbox{(O2})$ will be imposed in order to guarantee the \textit{Comparison Principle} for oblique derivative problems and consequently the classic Perron's Method for viscosity solutions(See \cite{Bessa} and \cite{BJ1}).

\begin{theorem}[{\bf Weighted Orlicz estimates for obstacle problems}]\label{T2}
Let $u$ be an $L^{p}$-viscosity solution of problem  \eqref{obss1}, where $p=p_{0}n$ , $F$ satisfies the structural assumptions (H1)-(H5), where $\beta,\gamma\in C^{2}(\partial \Omega)$ and (O1)-(O2), $\partial \Omega \in C^{3}$ and $\psi \in W^{2,\Upsilon}_{\omega}(\Omega)$, where $\Upsilon(t)=\Phi(t^{n})$ and $\omega\in \mathcal{A}_{i(\Phi)}$. Then, $u \in W^{2,\Upsilon}_{\omega}(\Omega)$  and 
\begin{equation*} \label{obs2}
\|u\|_{W^{2,\Upsilon}_{\omega}(\Omega)} \le \mathrm{C}\cdot \left( \|f\|_{L^{\Upsilon}_{\omega}(\Omega)}+ \|\psi\|_{WL^{2,\Upsilon}(\Omega)}  +\|g\|_{C^{1,\alpha}(\partial \Omega)} \right),
\end{equation*}
where $\mathrm{C}=\mathrm{C}(n,\lambda,\Lambda, p,p_{2},\mu_0, \sigma, \xi,\omega ,i(\Phi), \|\beta\|_{C^2(\partial \Omega)},\|\gamma\|_{C^{1,\alpha}(\partial \Omega)}, \partial \Omega, \textrm{diam}(\Omega), \theta_0)$.
\end{theorem}

The rest of this paper is organized as follows: Section \ref{section2} contains the main notations and preliminary results where stand out a review of the theory of weighted Orlicz spaces and approximation method of the recession operator. In Section \ref{section3}, we obtain interior and up to the boundary weighted Orlicz-Sobolev estimates and,with them, we present the proof of the Main Theorem \ref{T1} besides that we discuss about an improvement in the obtained estimates as well as we appreciate some applications of Theorem \ref{T1}. In the subsequent Section \ref{obstaculo}, we apply the  Theorem \ref{T1} in the context of the obstacle problem to obtain weighted Orlicz-Sobolev estimates of solutions to this problem. Finally, in the last Section \ref{section5} we address the limiting case $p = \infty$, that is, we provide weighted Orlicz-BMO a priori estimates for $D^{2}u$ in terms of the weighted Orlicz-BMO norm of the source term $f$ for the problem \eqref{1.1} without the dependence of terms of order 1 and 0, namely $Du$ and $u$.

\section{Preliminaries} \label{section2}

For the readers convenience, we collect some preliminaries related to fully nonlinear elliptic equations and viscosity solutions. In addition, we present here some results which will be used in the next section. 

We introduce the appropriate notions of viscosity solutions to 
\begin{equation}\label{2.1}
	\left\{
	\begin{array}{rclcl}
		F(D^2u,Du,u,x) &=& f(x)& \mbox{in} &   \Omega \\
		\mathfrak{B}(Du,u,x)&=& g(x) &\mbox{on}& \Gamma,
	\end{array}
	\right.
\end{equation}
where $\Gamma \subset \partial \Omega$ is a relatively open set. The viscosity solution concept emerged in the 1980s by Crandall and Lions in \cite{CL}. The following concept is a great tool to deal with problems of non-variational nature.
\begin{definition}[{\bf $C^{0}-$viscosity solutions}]\label{VSC_0} Let $F$ be a $\left(\lambda, \Lambda, \sigma, \xi\right)$-elliptic operator and let $f\in C^{0}(\Omega)$. A function $u \in C^{0}(\Omega \cup \Gamma)$  is said a $C^{0}-$viscosity solution of \eqref{2.1} if the following conditions hold:
\begin{enumerate}
\item[a)] for all $ \forall\,\, \varphi \in C^{2} (\Omega \cup \Gamma)$ touching $u$ by above  at  $x_0 \in \Omega \cup \Gamma$,
$$
\left\{
\begin{array}{rcl}
F\left(D^2 \varphi(x_{0}), D \varphi(x_{0}), \varphi(x_{0}), x_{0}\right)  \geq f(x_0) & \text{when} & x_0 \in \Omega \\
\mathfrak{B}(D\varphi(x_{0}),\varphi(x_{0}),x_{0}) \ge g(x_0) & \text{when} & x_0 \in \Gamma;
\end{array}
\right.
$$
\item[b)] for all $ \forall\,\, \varphi \in C^{2 } (\Omega \cup \Gamma)$ touching $u$ by below  at  $x_0 \in \Omega \cup \Gamma$,
$$
\left\{
\begin{array}{rcl}
F\left(D^2 \varphi(x_{0}), D \varphi(x_{0}), \varphi(x_{0}), x_{0}\right)  \leq f(x_0) & \text{when} & x_0 \in \Omega \\
\mathfrak{B}(D\varphi(x_{0}),\varphi(x_{0}),x_{0}) \le g(x_0) & \text{when} & x_0 \in \Gamma.
\end{array}
\right.
$$

\end{enumerate}
\end{definition}
\begin{definition}[{\bf$L^{p}$-viscosity solution}]\label{VSLp}
Let $F$ be a $(\lambda,\Lambda,\sigma, \xi)$-elliptic operator, $p>\frac{n}{2}$ and $f\in L^{p}(\Omega)$. Assume that $F$ is continuous in $\mathrm{X}$, $\zeta$, $r$ and measurable in $x$. A function $u\in C^{0}(\overline{\Omega})$ is said an $L^{p}$-viscosity solution for $\eqref{2.1}$ if the following assertions hold:
\begin{enumerate}
\item [a)] For all $\varphi\in W^{2,p}(\overline{\Omega})$ touching $u$ by above  at  $x_0 \in \overline{\Omega}$
$$
\left\{
\begin{array}{rcl}
F\left(D^2 \varphi(x_{0}), D \varphi(x_{0}), \varphi(x_{0}), x_{0}\right)  \geq f(x_0) & \text{when} & x_0 \in \Omega \\
\mathfrak{B}(D\varphi(x_{0}),\varphi(x_{0}),x_{0}) \ge g(x_0) & \text{when} & x_0 \in \partial \Omega;
\end{array}
\right.
$$
\item [b)] For all $\varphi\in W^{2,p}(\overline{\Omega})$ touching $u$ by below  at  $x_0 \in \overline{\Omega}$
$$
\left\{
\begin{array}{rcl}
F\left(D^2 \varphi(x_{0}), D \varphi(x_{0}), \varphi(x_{0}), x_{0}\right)  \leq f(x_0) & \text{when} & x_0 \in \Omega \\
\mathfrak{B}(D\varphi(x_{0}),\varphi(x_{0}),x_{0})\le g(x_0) & \text{when} & x_0 \in \partial \Omega.
\end{array}
\right.
$$
\end{enumerate}
\end{definition}

For convenience of notation, we define
$$
L^{\pm}(u) \defeq \mathcal{M}^{\pm}_{\lambda,\Lambda}(D^2 u) \pm \sigma |Du|.
$$

\begin{definition}
 We define the class $\overline{\mathcal{S}}\left(\lambda,\Lambda,\sigma, f\right)$ and $\underline{S}\left(\lambda,\Lambda,\sigma, f\right)$ to be the set of all continuous functions $u$ that satisfy $L^{+}(u) \ge f$, respectively $L^{-}(u) \le f$ in the viscosity sense (see Definition \ref{VSC_0}). We define
 $$
 \mathcal{S}\left(\lambda, \Lambda, \sigma,f\right) \defeq  \overline{\mathcal{S}}\left(\lambda, \Lambda, \sigma,f\right) \cap \underline{\mathcal{S}}\left(\lambda, \Lambda,\sigma, f\right)\,\,\text{and}\,\,
 \mathcal{S}^{\star}\left(\lambda, \Lambda,\sigma, f\right) \defeq  \overline{\mathcal{S}}\left(\lambda, \Lambda, \sigma,|f|\right) \cap \underline{\mathcal{S}}\left(\lambda, \Lambda,\sigma, -|f|\right).
  $$
  Moreover, when $\sigma=0$, we denote $\mathcal{S}^{\star}(\lambda,\Lambda,0,f)$ just by $S^{\star}(\lambda,\Lambda,f)$ (resp. by $\underline{S}, \overline{S}, \mathcal{S}$).
 \end{definition}

\begin{definition}
A paraboloid is a polynomial in $(x_1,\ldots,x_n)$ of degree $2$. We call $P$ a paraboloid of opening $M$ if
$$
	P_{M}(x)=a+b\cdot x\pm\frac{M}{2}|x|^{2}.
$$ When the sign of $P_{M}$ is "+" see that $P_{M}$ is a convex function and  a concave function otherwise.  
\end{definition}
Let $u$ be a continuous function in $\Omega$. We define, for any an open set $\widetilde{\Omega}\subset \Omega$ and $M>0$,
$$
\underline{G}_{\mathrm{M}}(u,\widetilde{\Omega}) \defeq \left\{x_0 \in \widetilde{\Omega} ; \exists \, \mathrm{P}_{\mathrm{M}} \,\,\, \textrm{concave paraboloid s. t.} \,\,\, \mathrm{P}_{\mathrm{M}}(x_0)=u(x_0), \,\, \mathrm{P}_{\mathrm{M}}(x) \le u(x)\,\, \forall \, x \in \widetilde{\Omega}\right\}
$$
and
$$
\underline{A}_{\mathrm{M}}(u,\widetilde{\Omega}) \defeq \widetilde{\Omega} \setminus \underline{G}_{\mathrm{M}}(u,\widetilde{\Omega}).
$$

Analogously we define , using concave paraboloids, $\overline{G}_{\mathrm{M}}(u,\widetilde{\Omega})$ and  $\overline{A}_{\mathrm{M}}(u,\widetilde{\Omega})$. Also define the sets
$$
G_{\mathrm{M}}(u,\widetilde{\Omega}) \defeq  \underline{G}_{\mathrm{M}}(u,\widetilde{\Omega}) \cap \overline{G}_{\mathrm{M}}(u,\widetilde{\Omega})\,\,\,\text{and}\,\,\,A_{\mathrm{M}}(u,\widetilde{\Omega}) \defeq \underline{A}_{\mathrm{M}}(u,\widetilde{\Omega}) \cap \overline{A}_{\mathrm{M}}(u,\widetilde{\Omega}).
$$
Furthermore, we define:
\begin{eqnarray*}
\overline{\Theta}(u,\widetilde{\Omega})(x) &\defeq& \inf\left\{\mathrm{M}>0 ; \, x \in \overline{G}_{\mathrm{M}}(u,\widetilde{\Omega})\right\},\\
\underline{\Theta}(u,\widetilde{\Omega})(x)&\defeq& \inf\left\{\mathrm{M}>0 ; \, x \in \underline{G}_{\mathrm{M}}(u,\widetilde{\Omega})\right\},\\
\Theta(u,\widetilde{\Omega})(x) &\defeq& \sup\left\{\underline{\Theta}(u,\widetilde{\Omega})(x), \overline{\Theta}(u,\widetilde{\Omega})(x)\right\}.
\end{eqnarray*}
For more details about such classes of solutions as well as the functions above that relate the tangent paraboloids to a continuous function we recommend \cite{CC} and \cite{CCKS}.

Now, we present the following Stability result (see \cite[Theorem 3.8]{CCKS} for a proof).
\begin{lemma}[{\bf Stability Lemma}]\label{Est}
For $k \in \mathbb{N}$ let $\Omega_k \subset \Omega_{k+1}$ be an increasing sequence of domains and $\displaystyle \Omega \defeq \bigcup_{k=1}^{\infty} \Omega_k$. Let $p > n$ and $F, F_k$ be $(\lambda, \Lambda, \sigma, \xi)-$elliptic operators. Assume $f \in L^{p}(\Omega)$, $f_k \in L^p(\Omega_k)$ and that $u_k \in C^0(\Omega_k)$ are $L^{p}-$viscosity sub-solutions (resp. super-solutions) of
$$
	F_k(D^2 u_k,Du_k,u_k,x)=f_k(x) \quad \textrm{in} \quad \Omega_k.
$$
Suppose that $u_k \to u_{\infty}$ locally uniformly in $\Omega$ and that for $\mathrm{B}_r(x_0) \subset \Omega$ and $\varphi \in W^{2,p}(\mathrm{B}_r(x_0))$ we have
\begin{equation} \label{Est1}
	\|(\hat{g}-\hat{g}_k)^+\|_{L^p(\mathrm{B}_r(x_0))} \to 0 \quad \left(\textrm{resp.} \,\,\, \|(\hat{g}-\hat{g}_k)^-\|_{L^p(\mathrm{B}_r(x_0))} \to 0 \right),
\end{equation}
where $\hat{g}(x) \defeq F(D^2 \varphi, D \varphi, u,x)-f(x)$ and $\hat{g}_k(x) =  F_k(D^2 \varphi, D \varphi, u_{k},x)-f_k(x)$.  Then $u$ is an $L^{p}-$viscosity sub-solution (resp. super-solution) of
$$
	F(D^2 u,Du,u,x)=f(x) \quad \textrm{in} \quad \Omega.
$$
Moreover, if $F$ and $f$ are continuous, then $u$ is a $C^0-$viscosity sub-solution (resp. super-solution) provided that \eqref{Est1} holds for all $\varphi \in C^2(\mathrm{B}_r(x_0))$ test function.
\end{lemma}

The next result is an ABP Maximum Principle (see \cite[Theorem 2.1]{LiZhang} for a proof).
\begin{lemma}[{\bf ABP Maximum Principle}]\label{ABP-fullversion}
Let $\Omega\subset \mathrm{B}_{1}$ and $u$ satisfying
\begin{equation*}
\left\{
\begin{array}{rclcl}
u\in \mathcal{S}(\lambda,\Lambda,f) &\mbox{in}&   \Omega \\
\mathfrak{B}(Du,u,x)=g(x)  &\mbox{on}&  \Gamma.
\end{array}
\right.
\end{equation*}
Suppose that exists $\varsigma\in \partial \mathrm{B}_{1}$ such that $\beta\cdot\varsigma\geq \mu_0$ and $\gamma\le 0$ in $\Gamma$. Then,
	\begin{eqnarray*}
		\|u\|_{L^{\infty}(\Omega)}\leq \|u\|_{L^{\infty}(\partial \Omega\setminus \Gamma)}+\mathrm{C}(n, \lambda, \Lambda, \mu_0)(\| f\|_{L^{n}(\Omega)}+\| g\|_{L^{\infty}(\Gamma)})
	\end{eqnarray*}
\end{lemma}

In the sequel, we comment on the existence and uniqueness of viscosity solutions with oblique boundary conditions. For that purpose, we will assume the following condition on $F$:

\begin{enumerate}
\item [(H)](\textbf{Modulus of continuity under coefficients}) There exists a modulus of continuity $\tilde{\omega}$, i.e., $\tilde{\omega}$ is nondecreasing with $ \displaystyle \lim_{t \to 0} \tilde{\omega}(t) =0$ and
$$
\psi_{F}(x,y) \le \tilde{\omega}(|x-y|).
$$
\end{enumerate}

 With this hypothesis we can assure existence and uniqueness of viscosity solutions to following problem:
$$
\left\{
\begin{array}{rclcl}
F(D^2u, x) &=& f(x) & \mbox{in} & \Omega,\\
\mathfrak{B}(Du,u,x)&=& g(x) & \mbox{on} & \Gamma\\
u(x)&=&\varphi(x) & \mbox{in} & \partial \Omega\setminus \Gamma,
\end{array}
\right.
$$
where $\Gamma$ is relatively open of $\partial\Omega$. In \cite{Bessa} can be seen a proof for the next theorem. For this reason, we will omit it here.

\begin{theorem}[{\bf Existence and Uniqueness}]\label{Existencia}
Let $\Gamma\in C^{2}$, $\beta\in C^{2}(\overline{\Gamma})$, $\gamma\le 0$ and $\varphi\in C^{0}(\partial \Omega\setminus \Gamma)$. Suppose that there exists $\varsigma\in\partial \mathrm{B}_{1}$ such that $\beta\cdot \varsigma\ge \mu_0$ on $\Gamma$ and assume $(H)$. In addition, suppose that $\Omega$ satisfies an exterior cone condition at any $x\in \partial \Omega\setminus \overline{\Gamma}$ and satisfies an exterior sphere condition at any $x\in \overline{\Gamma}\cap (\partial \Omega\setminus \Gamma)$. Then, there exists a unique viscosity solution $u\in C^{0}(\overline{\Omega})$ of
	$$
	\left\{
	\begin{array}{rclcl}
	F(D^2u, x) &=& f(x) & \mbox{in} & \Omega,\\
	\mathfrak{B}(Du,u,x)&=& g(x) & \mbox{on} & \Gamma\\
	u(x)&=&\varphi(x) & \mbox{in} & \partial \Omega\setminus \Gamma.
	\end{array}
	\right.
	$$
\end{theorem}
For a study of the existence and uniqueness of solutions to problems of an oblique nature, we recommend the classic work by Ishii and Lions \cite{Hi}.
\subsection{A few more facts about weights and weighted Orlicz spaces}

\hspace{0.4cm} We will present some more properties already known about weights and weighted Orlicz spaces that will be indispensable for the progress of this work. Initially, about weights we have the following lemma whose proof we found in Kokilashvili and Krbec's book  \cite{KokiMiro}. Next we will denote by $\Leb(A)$ the n-dimensional Lebesgue measure of the measurable set $A\subset \mathbb{R}^{n}$.

\begin{lemma}\label{propriedadesdospesos}
Let $\omega$ be an $A_{s}$ weight for some $s\in(1,\infty)$. Then,
\begin{enumerate}
\item[(a)](increasing) If $r\ge s$, then $\omega$ belong to class $A_{r}$ and $[\omega]_{r}\leq [\omega]_{s}$. 
\item[(b)] (open-end) There exists a small enough constant $\varepsilon_{0}>0$ depending  only on $n$, $s$ and $[\omega]_{s}$ such that $\omega\in A_{s-\varepsilon_{0}}$ with $s-\varepsilon_{0}>1$.
\item[(c)] (strong doubling) There exist two postive constants $k_{1}$ and $\theta\in(0,1)$ depending only on $n$, $s$ and $[\omega]_{s}$ such that 
\begin{eqnarray*}
\frac{1}{[\omega]_{s}}\left(\frac{\Leb(E)}{\Leb(\Omega)}\right)^{s}\leq \frac{\omega(E)}{\omega(\Omega)}\leq k_{1}\left(\frac{\Leb(E)}{\Leb(\Omega)}\right)^{\theta}.
\end{eqnarray*}  
for all $E\subset\Omega$ Lebesgue measurable set. 
\end{enumerate} 
\end{lemma}

We recall of Hardy-Littlewood maximal function. For $f\in L^{1}_{\loc}(\mathbb{R}^{n})$, the Hardy-Littlewood maximal function  de $f$ is defined by
\begin{eqnarray*}
	\mathcal{M}(f)(x)=\sup_{r>0}\intav{B_{r}(x)}|f(y)|dy.
\end{eqnarray*}
We have a version of classical Hardy-Littlewood-Wiener theorem for weighted Orlicz spaces(cf.\cite[Theorem 2.1.1]{KokiMiro} and ) that we will use later.
\begin{lemma}\label{maximalorlicz}
Let $\Phi$ be an $N$-function such that $\Phi\in \Delta_{2}\cap\nabla_{2}$ and $\omega\in \mathcal{A}_{i(\Phi)}$. Then  
\begin{eqnarray*}
\rho_{\Phi,\omega}(g)\leq \rho_{\Phi,\omega}\left(\mathcal{M}(g)\right)\leq \mathrm{C}\rho_{\Phi,\omega}(g)
\end{eqnarray*}
for all $g\in L^{\Phi}_{\omega}(\mathbb{R}^{n})$, where a constant $\mathrm{C}>0$ is idependent of $g$. 
\end{lemma} 

Later we will need a sufficient condition for the Hessian in the distributional sense to be in weighted Orlicz spaces. This is guaranteed from the following lemma (for proof cf. \cite[Lemma 3.4]{BLOK}).

\begin{lemma}\label{caracterizationofhessian}
Let $\Phi$ be an $N$-function satisfying $\Delta_{2}\cap \nabla_{2}$-condition and $\omega\in \mathcal{A}_{i(\Phi)}$ weight. Assume that $u\in C^{0}(E)$ for bounded domain $E\subset \mathbb{R}^{n}$ and set for $r>0$
\begin{eqnarray*}
\Theta(u,r)(x)\defeq  \Theta(u,B_{r}(x)\cap E)(x), \ x\in E.
\end{eqnarray*}
If $\Theta(u,r)\in L^{\Phi}_{\omega}(E)$, then Hessian in the distributional sense $D^{2}u\in L^{\Phi}_{\omega}(E)$ with estimate
\begin{eqnarray*}
\|D^{2}u\|_{L^{\Phi}_{\omega}(E)}\leq 8\|\Theta(u,r)\|_{L^{\Phi}_{\omega}(E)}.
\end{eqnarray*} 
\end{lemma}

Finally, we will need the following standard characterization of functions in weighted Orlicz spaces whose proof goes back to standard measure theory arguments(cf. \cite[Lemma 4.6]{BOKPS}).
\begin{proposition}\label{caracterizacaodosespacosdeorliczcompeso}
Let $\Phi\in \Delta_{2}\cap \nabla_{2}$ be an $N$-function and $\omega$ an $A_{s}$ weight for some $s\in(1,\infty)$, $g: E \to \R$ be a nonnegative measurable function in a bounded domain $E\subset \mathbb{R}^{n}$. Let $\eta>0$ and $M >1$ constants. Then, 
$$
g \in L^{\Phi}_{\omega}(E) \Longleftrightarrow  \sum_{j=1}^{\infty} \Phi(M^{j}) \omega(\{x\in E; g(x)>\eta M^{j}\})\defeq S< \infty
$$
and 
\begin{eqnarray*} 
\mathrm{C}^{-1}S\le \rho_{\Phi,\omega}(g) \le \mathrm{C}(\omega(E)+S),
\end{eqnarray*}
with $\mathrm{C}=\mathrm{C}(\eta, \mathrm{M},\Phi,\omega)$ is positive constant. 
\end{proposition}

With the Proposition \ref{caracterizacaodosespacosdeorliczcompeso} and the Lemma \ref{caracterizationofhessian}, we will study $\mathcal{L}^n(\mathcal{A}_t)$ has a power decay in $t$ for $u \in \mathcal{S}(\lambda,\Lambda,f)$ since  the distribution function of $\Theta(u,r)$  can be controlled by such measures associated with the sets $\mathcal{A}_{t}$. Second, we use an approximation to acelerate the power decay corresponding to $F(D^2u,x)=f$ with oblique boundary condition.  We gather few elements involved in the proof of Theorem \ref{T1}, although such results are well- known facts in the literature. Thus, the proof is omitted in what follows. For more details see \cite{Bessa}
\begin{proposition}
Let $u \in \mathcal{S}(\lambda,\Lambda, f)$ in $B^+_{12 \sqrt{n}} \subset \Omega \subset \mathbb{R}^n_+$,  $u \in C^{0}(\Omega)$ and $\|u\|_{L^{\infty}(\Omega)} \le 1$. Then, there exist universal constant $C>0$ and $\delta >0$ such that if $\|f\|_{L^{n}(B^+_{12 \sqrt{n}})} \le 1$ implies
$$
	\mathcal{L}^n \left( \mathcal{A}_t(u,\Omega) \cap \left( \mathcal{Q}^{n-1}_1 \times (0,1) + x_0\right)\right) \le C \cdot t^{-\delta}
$$
for any $x_0 \in B_{9 \sqrt{n}} \cap \overline{\mathbb{R}^n_+}$ and $t >1$.
\end{proposition}

\hspace{0.4cm} In this part we will present a key tool that enables a compactness method via an asymptotic recession operator. This ensures the possibility of weakening the convexity assumption on the governing operator $F$ by imposing conditions on $F^{\sharp}$. Roughly speaking, the following result tells us that if our equation is close enough to the homogeneous equation with constant coefficients, our solution will be close enough to a solution of the homogeneous equation with frozen coefficients. For more details of this result we recommend the proof in \cite[Lemma 2.12]{Bessa},


\begin{lemma}[{\bf Approximation}] \label{Approx}
Let $n \le p < \infty$, $0 \le \tilde{\nu} \le 1$ and assume that $(H1)-(H3)$ and $(H5)$ are in force. Then, for every $\delta >0$, $\varphi \in C(\partial \mathrm{B}_1(0^{\prime},\tilde{\nu}))$ with $\|\varphi\|_{L^{\infty}(\partial \mathrm{B}_1(0^{\prime},\tilde{\nu}))} \le \mathfrak{c}_1$ and $g \in C^{1,\alpha}(\overline{\mathrm{T}}_2)$ with $0 < \alpha < 1$  and $\|g\|_{C^{1,\alpha}(\overline{\mathrm{T}}_2)} \le \mathfrak{c}_2$ for some $\mathfrak{c}_2 >0$ there exist positive constants $\epsilon =\epsilon(\delta,n, \mu_0, p, \lambda, \Lambda, \gamma,\mathfrak{c}_1, \mathfrak{c}_2) < 1$ and $\tau_0 = \tau_0(\delta, n, \lambda,\Lambda, \mu_0, \mathfrak{c}_1, \mathfrak{c}_2) >0$ such that, if
$$
\max\left\{ |F_{\tau}(\mathrm{X},x) - F^{\sharp}(\mathrm{X},x)|, \, \|\psi_{F^{\sharp}}(\cdot,0)\Vert_{L^{p}(\mathrm{B}^{+}_{r})},\,\|f\|_{L^{p}(\mathrm{B}^+_{r})}  \right\} \le \epsilon \quad \textrm{and} \quad \tau \le \tau_0
$$
then, any two $L^p$-viscosity solutions $u$ (normalized, i.e., $\|u\|_{L^{\infty}(\mathrm{B}^{+}_r(0^{\prime},\tilde{\nu}))}\le 1$) and $\mathfrak{h}$ of
$$
\left\{
\begin{array}{rclcl}
F_{\tau}(D^2u,x) &=& f(x)& \mbox{in} & \mathrm{B}^{+}_r(0^{\prime},\tilde{\nu}) \cap \mathbb{R}^{n}_+ \\
\mathfrak{B}(Du,u,x)&=& g(x) & \mbox{on} &  \mathrm{B}_r(0^{\prime},\tilde{\nu}) \cap \mathrm{T}_r\\
u(x) &=& \varphi(x) &\mbox{on}& \overline{\partial \mathrm{B}_r(0^{\prime},\tilde{\nu}) \cap \mathbb{R}^n_+}
\end{array}
\right.
$$
and
$$
\left\{
\begin{array}{rclcl}
F^{\sharp}(D^2 \mathfrak{h},0) &=& 0& \mbox{in} & \mathrm{B}^{+}_{\frac{7}{8}r}(0^{\prime},\tilde{\nu}) \cap \mathbb{R}^n_+ \\
\mathfrak{B}(D\mathfrak{h},\mathfrak{h},x) &=& g(x) & \mbox{on} &  \mathrm{B}_{\frac{7}{8} r}(0^{\prime},\tilde{\nu}) \cap \mathrm{T}_r\\
\mathfrak{h}(x) &=& u(x) &\mbox{on}& \overline{\partial \mathrm{B}_{\frac{7}{8}r}(0^{\prime}, \tilde{\nu}) \cap \mathbb{R}^n_+}		
\end{array}
\right.
$$
satisfy
$$
\|u-\mathfrak{h}\|_{L^{\infty}(\mathrm{B}^{+}_{\frac{7}{8}r}(0^{\prime},\tilde{\nu}))} \le \delta.
$$
\end{lemma}

With this version of approximation lemma we obtain the following lemma(cf.\cite[Lemma 3.6]{Bessa}). 

\begin{lemma} \label{lemma2.17}
	Given $\epsilon_0  \in (0, 1)$ and let $u$ be a normalized viscosity solution to
	\begin{equation*}
		\left\{
		\begin{array}{rclcl}
			F_{\tau}(D^2u,x) &=& f(x) & \mbox{in} & \mathrm{B}^+_{14\sqrt{n}},\\
			\mathfrak{B}(Du,u,x)&=& g(x) & \mbox{on} & \mathrm{T}_{14 \sqrt{n}}.
		\end{array}
		\right.
	\end{equation*}
Assume that $(H1)-(H3)$ and $(H5)$  hold and extend $f$ by zero outside $\mathrm{B}^+_{14\sqrt{n}}$. For $x \in \mathrm{B}_{14 \sqrt{n}}$, let
	$$
	\max \left\{\tau, \|f\|_{L^n(\mathrm{B}_{14 \sqrt{n}})} \right\} \le \epsilon
	$$
	for some $\epsilon >0$ depending only $n, \epsilon_0, \lambda, \Lambda, \mu_0, \alpha$. Then, for $k \in \mathbb{N}\setminus \{0\}$ we define
	\begin{eqnarray*}
		\mathcal{A} & \defeq & \mathcal{A}_{\mathrm{M}^{k+1}}(u, \mathrm{B}^+_{14 \sqrt{n}}) \cap \left(\mathcal{Q}^{n-1}_1 \times (0,1)\right)\\
		\mathcal{B} &\defeq & \left(\mathcal{A}_{\mathrm{M}^k}(u, \mathrm{B}^+_{14\sqrt{n}}) \cap \left(Q^{n-1}_1 \times (0,1)\right)\right)\cup \left\{x \in \mathcal{Q}^{n-1}_1 \times (0,1); \mathcal{M}(f^n) \ge (\mathrm{C}_0\mathrm{M}^k)^n \right\},
	\end{eqnarray*}
where $\mathrm{C}_{0}\ll1$ and $\mathrm{M} = \mathrm{M}(n, \mathrm{C}_0)>1$. Then,
$$
\Leb(\mathcal{A})\le \epsilon_0\Leb(\mathcal{B}).
	$$
\end{lemma}

\begin{lemma}\label{lema2.15}
Under the same conditions of Lemma \ref{lemma2.17} and assuming  that $\omega\in A_{p}$ for some $1<p<\infty$, fix $\epsilon_{0}\in(0,1)$. Now, for $k \geq 0$ we define
\begin{eqnarray*}
	A^{k}\defeq  \mathcal{A}_{\mathrm{M}^k}(u, \mathrm{B}^+_{14\sqrt{n}}) \cap \left(\mathcal{Q}^{n-1}_1 \times (0,1)\right) \  \mbox{and} \ 
	B^{k} \defeq  \left\{ x \in \left(\mathcal{Q}^{n-1}_1 \times (0,1)\right); \mathcal{M}(f^n)(x) \ge (\mathrm{C}_0\mathrm{M}^k)^n\right\},
\end{eqnarray*}
where $\mathrm{C}_{0}$ and $\mathrm{M}$ are the constants in Lemma \ref{lemma2.17}. Then, for any $k\geq 0$,
\begin{eqnarray*}
\omega\left(A^{k}\right)\leq \epsilon_{0}^{k}\omega(A^{0})+\sum_{i=1}^{k-1}\epsilon_{0}^{k-i}\omega(B^{i}).
\end{eqnarray*}
\end{lemma}
\begin{proof}
Really, let us $\epsilon_{0}\in(0,1)$. Applying the lemma \ref{lemma2.17} to the constant $\tilde{\epsilon}_{0}=\left(\frac{\epsilon_{0}}{k_{1}}\right)^{\frac{ 1}{\theta}}$, where $\theta$ and $k_{1}$ are the constants of Lemma \ref{propriedadesdospesos}, we obtain the following estimate
\begin{eqnarray*}
\Leb(A^{k+1})\leq \tilde{\epsilon_{0}}\Leb(A^{k}\cup B^{k})
\end{eqnarray*}
and consequently by strong doubling ( item (c) of Lemma \ref{propriedadesdospesos})
\begin{eqnarray*}
\omega(A^{k+1})\leq k_{1}\left( \frac{\Leb(A^{k+1})}{\Leb(A^{k}\cup B^{k})}\right)^{\theta}\omega(A^{k}\cup B^{k})= \epsilon_{0}\omega(A^{k}\cup B^{k})\le \epsilon_{0}\omega(A^{k})+\epsilon_{0}\omega(B^{k}), \forall k\geq 0.
\end{eqnarray*}
By iterating over these estimates, we obtain the desired result.
\end{proof}

\section{Weighted Orlicz regularity for the problem \eqref{1.1}} \label{section3}

\hspace{0.4cm} In this part we aim to prove Theorem  \ref{T1}. For this, we will first study the local problem given by
\begin{equation} \label{equation1}
\left\{
\begin{array}{rclcl}
F(D^2u,Du,u,x) &=& f(x) & \mbox{in} & \mathrm{B}^+_1,\\
\mathfrak{B}(Du,u,x)&=& g(x) & \mbox{on} & \mathrm{T}_1.
\end{array}
\right.
\end{equation}
where we will show under certain conditions weighted Orlicz estimates for the viscosity solutions of this problem. With this fact in hand, we can, via density and standard coverage argument, guarantee a proof of the Theorem \ref{T1}.

Initially we have for $C^{0}$-viscosity solutions of \eqref{equation1} without the dependence of $Du$ and $u$ the following result:

\begin{proposition}\label{T-flat}
Let $\Phi\in \Delta_{2}\cap\nabla_{2}$ be an $N$-function, $f \in L^{\Upsilon}_{\omega}(\mathrm{B}^+_1)\cap C^{0}(\mathrm{B}^{+}_{1})$ where $\omega\in \mathcal{A}_{i(\Phi)}$ be an weight and $\Upsilon(t)=\Phi(t^{n})$. Let $u$ be a normalized $C^{0}$-viscosity solution of
\begin{equation*} \label{mens}
\left\{
\begin{array}{rclcl}
F(D^2u, x) &=& f(x) & \mbox{in} & \mathrm{B}^+_1,\\
\mathfrak{B}(Du,u,x)&=& g(x) & \mbox{on} & \mathrm{T}_1.
\end{array}
\right.
\end{equation*}
 Assume that assumptions (H1)-(H5) are in force. Then $D^{2}u \in L^{\Upsilon}_{\omega}\left(\mathrm{B}^+_{\frac{1}{2}}\right)$ and
$$
\|D^{2}u\|_{L^{\Upsilon}_{\omega}\left(\mathrm{B}^+_{\frac{1}{2}}\right)} \le \mathrm{C} \cdot \left( \|u\|^{n}_{L^{\infty}(\mathrm{B}^+_1)} + \|f\|_{L^{\Upsilon}_{\omega}(\mathrm{B}^+_1)}+\Vert g\Vert_{C^{1,\alpha}(\overline{\mathrm{T}_{1}})}\right),
$$
where $\mathrm{C}=\mathrm{C}(n,\lambda,\Lambda,i(\Phi), p_{2}, \omega, \Vert \beta\Vert_{C^{1,\alpha}(\overline{\mathrm{T}_{1}})}, \Vert \gamma\Vert_{C^{1,\alpha}(\overline{\mathrm{T}_{1}})}, \alpha,r_{0},\theta_0,\mu_{0})>0$.
\end{proposition}

\begin{proof}
Initially, let's observe that $\Upsilon$ is an $N$-function that satisfies $\Upsilon\in\Delta_{2}\cap\nabla_{2}$, since $\Phi$ fulfills such conditions. Furthermore, it is interesting to observe that the constants of $\Upsilon$ for the condition $\Delta_{2}\cap\nabla{2}$ are the same as for $\Phi$. However, it is not difficult to show that $i(\Upsilon)=ni(\Phi)$ and thus $\mathcal{A}_{i(\Phi)}\subset \mathcal{A}_{i(\Upsilon) }$ by item (a) of Lemma \ref{propriedadesdospesos}. This guarantees, for example, that the Lemma \ref{mergulhoorliczlebesgue} applied to $L^{\Upsilon}_{\omega}$ depends only on $n$, $\omega$ and $i(\Phi)$ still. After this digression, we prove what we want.

Fix $x_0 \in \mathrm{B}^{+}_{\frac{1}{2}}\cup T_{\frac{1}{2}}$. When $x_0 \in \mathrm{T}_{\frac{1}{2}}$, we choosen $0 < r < \frac{1-|x_0|}{14\sqrt{n}}$ and define
$$
\kappa \defeq \frac{\epsilon r}{(\epsilon^{n} r^{-n} \|u\|_{L^{\infty}(\mathrm{B}^+_{14r\sqrt{n}}(x_0))}^{n}  + (\mathrm{C}')^{n}\|f\|_{L^{\Upsilon}_{\omega}(\mathrm{B}^+_{14r\sqrt{n}}(x_0))}+\epsilon^{n} r^{-n}\Vert g\Vert_{C^{1,\alpha}(\mathrm{T}_{14r\sqrt{n}}(x_{0}))})^{\frac{1}{n}}}
$$
where the constant $\epsilon>0$ of Lemma \ref{lemma2.17} for $\epsilon_{0}>0$  constant will be determined a posteriori and $\mathrm{C}'$ is the constant of Lemma \ref{mergulhoorliczlebesgue}. By choosen of $r$ we can define $\tilde{u}(y) \defeq \frac{\kappa}{r^{2}}u(x_0+ry), y\in B^{+}_{14\sqrt{n}}$. Observe that, $\tilde{u}$ is a normalized $C^{0}$-viscosity solution to
$$
\left\{
\begin{array}{rclcl}
\tilde{F}(D^2 \tilde{u},x) &=& \tilde{f}(x) & \mbox{in} & \mathrm{B}^+_{14\sqrt{n}},\\
\tilde{\mathfrak{B}}(D\tilde{u},\tilde{u},x)
&=&\tilde{g}(x)  & \mbox{on} & \mathrm{T}_{14\sqrt{n}} .
\end{array}
\right.
$$
where
$$
\left\{
\begin{array}{rcl}
\tilde{F}(\mathrm{X}, y) &\defeq& \kappa F\left(\frac{1}{\kappa} \mathrm{X},x_{0}+ry\right) \\
\tilde{f}(y) &\defeq& \kappa f(x_0+ry)\\
\tilde{\mathcal{B}}(\overrightarrow{v},s,y)&\defeq& \tilde{\beta}(y)\cdot \overrightarrow{v}+\tilde{\gamma}(y)s\\
\tilde{\beta}(y) &\defeq&  \beta(x_0+ry)\\
\tilde{\gamma}(y) &\defeq& r\gamma(x_{0}+ry)\\
\tilde{g}(y)&\defeq& \frac{\kappa}{r}g(x_{0}+ry)\\
\tilde{\omega}(y)&\defeq&\omega(x_{0}+ry).
\end{array}
\right.
$$
Hence, $\tilde{F}$ fulfills the conditions (H1)-(H5) and $\tilde{\omega}\in \mathcal{A}_{i(\Phi)}$(since $\omega\in \mathcal{A}_{i(\Phi)}$). Furthermore, by Lemma \ref{mergulhoorliczlebesgue} and Hölder's inequality it follows that
\begin{eqnarray*}
\|\tilde{f}\|_{L^{n}\left(\mathrm{B}^+_{14\sqrt{n}}\right)} &=&\frac{\kappa}{r}\|f\|_{L^{n}(\mathrm{B}^{+}_{14r\sqrt{n}}(x_{0}))}\leq \frac{\kappa}{r}\mathrm{C}'\|f\|_{L^{\Upsilon}_{\omega}(\mathrm{B}^{+}_{14r\sqrt{n}}(x_{0}))}^{\frac{1}{n}}  \le \epsilon.
\end{eqnarray*}
Thus, we are in the hypotheses of Lemma \ref{lema2.15} and hence denoting for each $k\geq 0$,
\begin{eqnarray*}
A^{k}\defeq \mathcal{A}_{\mathrm{M}^k}(\tilde{u}, \mathrm{B}^+_{14\sqrt{n}}) \cap \left(\mathcal{Q}^{n-1}_1 \times (0,1)\right) \ \mbox{and} \ 
B^{k} \defeq \left\{ x \in \left(\mathcal{Q}^{n-1}_1 \times (0,1)\right);\mathcal{M}(\tilde{f}^n)(x) \ge (\mathrm{C}_0\mathrm{M}^k)^n\right\}
\end{eqnarray*}
we obtain
\begin{eqnarray}\label{estimativa2daprop3.1}
\tilde{\omega}(A^{k})\leq \epsilon_{0}^{k}\tilde{\omega}(A^{0})+\displaystyle\sum_{i=1}^{k-1}\epsilon_{0}^{k-i}\tilde{\omega}(B^{i}).
\end{eqnarray}
On the other hand, 
by the hypothesis that $f\in L^{\Upsilon}_{\omega}(\mathrm{B}^{+}_{1})$, we have that $|\tilde{f}|^{n}\in L^{\Phi}_{\omega}(\mathrm{B}^{+}_{14\sqrt{n}})$ and consequently by Lemma \ref{maximalorlicz} and \eqref{estimativadamodular},
\begin{eqnarray*}
\rho_{\Phi,\tilde{\omega}}(\mathcal{M}(|\tilde{f}|^{n}))&\leq&\mathrm{C}\rho_{\Phi,\tilde{\omega}}(|\tilde{f}|^{n})=\frac{\mathrm{C}}{r^{n}}\int_{\mathrm{B}^{+}_{14r\sqrt{n}}(x_{0})}\Phi\left(\kappa^{n}|f(y)|^{n}\right)\omega(y)dy\\
&\leq& \frac{\mathrm{C}}{r^{n}}\left(\left\|(\kappa f)^{n}\right\|_{L^{\Phi}_{\omega}(\mathrm{B}^{+}_{14\sqrt{n}}(x_{0}))}^{p_{2}}+1\right)=\frac{\mathrm{C}}{r^{n}}\left(\kappa^{np_{2}}\|f\|_{L^{\Upsilon}_{\omega}(\mathrm{B}^{+}_{14r\sqrt{n}}(x_{0}))}^{p_{2}}+1\right)\\
&\leq&\frac{\mathrm{C}}{r^{n}}((\epsilon r)^{np_{2}}+1)\leq \mathrm{C}, 
\end{eqnarray*}
so we have $\mathcal{M}(|\tilde{f}|^{n})\in L^{\Phi}_{\omega}(\mathrm{B}^{+}_{14\sqrt{n}})$ and still worth 
\begin{eqnarray}\label{estimativa3daprop3.1}
\|\mathcal{M}(|\tilde{f}|^{n})\|_{L^{\Phi}_{\omega}(\mathrm{B}^{+}_{14\sqrt{n}})}\leq \mathrm{C}
\end{eqnarray}
Thus, by $\Phi\in \Delta_{2}$, there exists constant $\mathrm{K}_{0}=\mathrm{K}_{0}(\mathrm{M}^{n})>0$ such that, $\Phi(\mathrm{M}^{n}t)\leq \mathrm{K}_{0}\Phi(t)$ for all $t\geq 0$. consequently, by iteration it follows that for every $k\in \mathbb{N}$ the following inequalities are valid $\Phi(M^{kn})\leq\mathrm{K}_{0}^{k}\Phi(1)$ and $\Phi(\mathrm{M}^{kn})\leq \mathrm{K}_{0}^{k-i}\Phi(\mathrm{M}^{in})$, for all $1\leq i\leq k-1$. 
Therefore, using \eqref{estimativa2daprop3.1} and \eqref{estimativa3daprop3.1}, we obtain
\begin{eqnarray*}
\sum_{k=1}^{\infty}\Upsilon(\mathrm{M}^{k})\tilde{\omega}(A^{k})&=&\sum_{k=1}^{\infty} \Phi(\mathrm{M}^{kn}) \tilde{\omega}(\mathcal{A}^{k}) \le \sum_{k=1}^{\infty}\Phi(1)\mathrm{K}_{0}^{k}\epsilon_{0}^{k}\tilde{\omega}(A^{0})+\sum_{k=1}^{\infty} \Phi(\mathrm{M}^{kn})\sum_{i=1}^{k-1}\epsilon^{k-i}_{0}\tilde{\omega}(B^{i})\\
&\leq& \Phi(1)\omega(\mathcal{Q}^{n-1}_{1}\times(0,1))\sum_{k=1}^{\infty}(\mathrm{K}_{0}\epsilon)^{k}+\sum_{k=1}^{\infty} \sum_{i=1}^{k-1}(\Phi(\mathrm{M}^{in})(\mathrm{K}_{0}\epsilon)^{k-i} \tilde{\omega}(B^{i})\\
&=& \sum_{k=1}^{\infty}(\mathrm{K}_{0}\epsilon_{0})^{k}\left(\Phi(1)\tilde{\omega}(\mathcal{Q}^{n-1}_{1}\times(0,1))+\sum_{i=1}^{\infty}\Phi(\mathrm{M}^{in})\tilde{\omega}(B^{i})\right)< +\infty,
\end{eqnarray*}
for $\epsilon_0$ small enough such that $\mathrm{K}_{0}\epsilon_{0}<1$.  This estimate together with the Proposition \ref{caracterizacaodosespacosdeorliczcompeso} implies that $\Theta(\tilde{u},B^{+}_{\frac{1}{2}})\in L^{\Upsilon}_{\tilde{\omega}}(B^{+}_{\frac{1}{2}})$, since, $\left\{x\in \mathrm{B}^{+}_{\frac{1}{2}}: \Theta(\tilde{u},\mathrm{B}^{+}_{\frac{1}{2}})(x)>\mathrm{M}^{k}\right\}\subset \mathcal{A}_{\mathrm{M}^{k}}(\tilde{u},\mathrm{B}^{+}_{\frac{1}{2}})$. Consequently by Lemma \ref{caracterizationofhessian}, $\|D^{2}\tilde{u}\|_{L^{\Upsilon}_{\tilde{\omega}}(\mathrm{B}^{+}_{\frac{1}{2}})}\le C$, for positive constant $C$ depending only on $n$, $\lambda$, $\Lambda$, $i(\Phi)$, $\omega$, $\|\beta\|_{C^{1,\alpha}(\overline{\mathrm{T}_{1}})}$, $\|\gamma\|_{C^{1,\alpha}(\overline{\mathrm{T}_{1}})}$, $\mu_0$, $r_{0}$ and $\theta_0$(constants of condition (H3)). Rescaling $u$, we get   
\begin{eqnarray*}
\|D^{2}u\|_{L^{\Upsilon}_{\omega}\left(\mathrm{B}^{+}_{\frac{r}{2}}(x_{0})\right)}\leq C(\|u\|^{n}_{L^{\infty}(\mathrm{B}^{+}_{1})}+\|f\|_{L^{\Upsilon}_{\omega}(\mathrm{B}^{+}_{1})}+\|g\|_{C^{1,\alpha}(\overline{\mathrm{T}_{1}})}),
\end{eqnarray*}
where $C>0$ depends only on $n$, $\lambda$, $\Lambda$, $i(\Phi)$, $p_{2}$, $\omega$,  $\|\beta\|_{C^{1,\alpha}(\overline{\mathrm{T}_{1}})}$, $\|\gamma\|_{C^{1,\alpha}(\overline{\mathrm{T}_{1}})}$, $\mu_{0}$, $r_{0}$, $\theta_0$ and $r$.

On the other hand, if $x_0 \in \mathrm{B}^+_{\frac{1}{2}}$, then by hypothesis (H4) we can obtain interior estimates by \cite[Theorem 2.4]{Lee} in form, 
\begin{eqnarray*}
\|D^{2}u\|_{L^{\Upsilon}_{\omega}\left(\mathrm{B}_{\frac{r}{2}}(x_{0})\right)}\leq C(\|u\|^{n}_{L^{\infty}(\mathrm{B}^{+}_{1})}+\|f\|_{L^{\Upsilon}_{\omega}(\mathrm{B}^{+}_{1})}).
\end{eqnarray*}
Finally, by combining interior and boundary estimates, we obtain the desired results by using a standard covering argument. This finishes proof of  the Proposition. 
\end{proof}


With this proposition in hand, we can find weighted Orlicz-Sobolev estimates for $C^{0}$-viscosity solutions of \eqref{equation1} by the following result.

\begin{proposition}\label{solution4entradas}
Let $u$ be a bounded $C^{0}$-viscosity solution of \eqref{equation1}. Asssume the structural conditions $(H1)-(H5)$,   $f\in L^{\Upsilon}_{\omega}(B^{+}_{1})$ for $\Upsilon(t)=\Phi(t^{n})$ and $\omega\in A_{i(\Phi)}$. Then, there exists a constant $\mathrm{C}>0$ depending on  $n$, $\lambda$, $\Lambda$, $\xi$, $\sigma$, $p_{2}$, $i(\Phi)$, $\omega$, $\Vert \beta\Vert_{C^{1,\alpha}(\overline{\mathrm{T}_{1}})}$, $\Vert \gamma\Vert_{C^{1,\alpha}(\overline{\mathrm{T}_{1}})}$, $\alpha$, $r_{0}$, $\theta_0$ and $\mu_{0}$, such that  $u \in W^{2,\Upsilon}_{\omega} \left(B^+_{\frac{1}{2}}\right)$ and
$$	\|u\|_{W^{2,\Upsilon}_{\omega}\left(\mathrm{B}^+_{\frac{1}{2}}\right)} \le \mathrm{C} \cdot \left( \|u\|^{n}_{L^{\infty}(\mathrm{B}^{+}_1)} + \|f\|_{L^{\Upsilon}_{\omega}(\mathrm{B}^{+}_1)}+\Vert g\Vert_{C^{1,\alpha}(\overline{\mathrm{T}_{1}})}\right).
$$
\end{proposition}

\begin{proof}
Notice that $u$ is also a viscosity solution of
\begin{equation*}
\left\{
\begin{array}{rclcl}
\tilde{F}(D^2u, x) &=& \tilde{f}(x) & \mbox{in} & \mathrm{B}^+_1,\\
\mathfrak{B}(Du,u,x)&=& g(x) & \mbox{on} & \mathrm{T}_1 .
\end{array}
\right.
\end{equation*}
where $\tilde{F}(\mathrm{X},x)  \defeq  F(\mathrm{X},0,0,x)$ and $\tilde{f}$ is a function satisfying
\begin{eqnarray*}
|\tilde{f}|\leq \sigma|Du|+\xi|u|+|f|.
\end{eqnarray*}
Thus, we can apply Proposition \ref{T-flat} and conclude that
\begin{eqnarray}\label{mens2}
\|D^{2}u\|_{L^{\Upsilon}_{\omega}\left(\mathrm{B}^{+}_{\frac{1}{2}}\right)} \le \mathrm{C} \cdot \left( \|u\|^{n}_{L^{\infty}(\mathrm{B}^{+}_{1})} + \|\tilde{f}\|_{L^{\Upsilon}_{\omega}(\mathrm{B}^+_1)}+\Vert g\Vert_{C^{1,\alpha}(\overline{\mathrm{T}_{1}})}\right).
\end{eqnarray}
Now, by proceeding similarly as in \cite[Corollary 4.5]{BJ} and \cite[Lemma 3.2]{BLOK} for obtain gradient estimate,
\begin{equation}\label{mens3}
\Vert Du\Vert_{L^{\Upsilon}_{\omega}(\mathrm{B}_{\frac{1}{2}}^{+})}\leq \mathrm{C}\cdot (\Vert u\Vert^{n}_{L^{\infty}(\mathrm{B}^{+}_{1})}+\Vert f\Vert_{L^{\Upsilon}_{\omega}(\mathrm{B}^{+}_{1})}+\Vert g\Vert_{C^{1,\alpha}(\overline{\mathrm{T}_{1}})}).
\end{equation}
Finally, combining the estimates (\ref{mens2}) and (\ref{mens3})  and the fact that $u\in L^{\Upsilon}_{\omega}(\mathrm{B}^{+}_{\frac{1}{2}})$ with estimate $\|u\|_{L^{\Upsilon}_{\omega}(\mathrm{B}^{+}_{\frac{1}{2}})}\leq \mathrm{C}\|u\|^{n}_{L^{\infty}(\mathrm{B}^{+}_{1})}$ we complete the desired estimate.
\end{proof}

Consequently, we have by classical arguments of density (cf. \cite{Bessa} and \cite{Winter}) the extension of the Proposition \ref{solution4entradas} to $L^{p}$-viscosity solutions. Due to the nature of the estimates and the Sobolev embeddings the most suitable and interesting $L^{p}$-viscosity solution environment is for $p=p_{0}n>n$, where $p_{0}$ is the constant of Lemma \ref{mergulhoorliczlebesgue}.

\begin{proposition}\label{Casolp}
Let $u$ be a bounded $L^{p}$-viscosity solution of \eqref{equation1} for $p=:p_{0}n\in(n,+\infty)$, where  $\beta,\gamma, g \in C^{1,\alpha}(\mathrm{T}_1)$ with $\beta \cdot \nu \ge \mu_0$ for some $\mu_0 >0$, $\gamma \le 0$, $f \in L^{\Upsilon}_{\omega}(\mathrm{B}^+_1)$, for $\Upsilon(t)=\Phi(t^{n})$ with $\Phi\in\Delta_{2}\cap\nabla_{2}$ be an $N$-function and $\omega\in A_{i(\Phi)}$. Further, assume that  $F^{\sharp}$ satisfies (H4)-(H5) and $F$ fulfills the condition $(H)$. Then, there exist positive constants $\beta_0=\beta_0(n,\lambda,\Lambda,p_{0},p_{2})$, $r_0 = r_0(n, \lambda, \Lambda, p_{0},p_{2})$ and $\mathrm{C}=\mathrm{C}(n,\lambda,\Lambda,\xi,\sigma,p_{0},p_{2},i(\Phi),\omega,\theta_0,\|\beta\|_{C^{1,\alpha}(\overline{\mathrm{T}_{1}})},\|\gamma\|_{C^{1,\alpha}(\overline{\mathrm{T}_{1}})},r_0)>0$, such that if
$$
\left(\intav{\mathrm{B}_r(x_0) \cap \mathrm{B}^+_1} \psi_{F^{\sharp}}(x, x_0)^{p} dx\right)^{\frac{1}{p}} \le \psi_0
$$
for any $x_0 \in \mathrm{B}^+_1$ and $r \in (0, r_0)$, then $u \in W^{2,\Upsilon}_{\omega}\left(\mathrm{B}^+_{\frac{1}{2}}\right)$ and
$$
\|u\|_{W^{2,\Upsilon}_{\omega}\left(\mathrm{B}^+_{\frac{1}{2}}\right)} \le \mathrm{C} \cdot\left( \|u\|^{n}_{L^{\infty}(\mathrm{B}^+_1)} +\|f\|_{L^{\Upsilon}_{\omega}(\mathrm{B}^+_1)}+\Vert g\Vert_{C^{1,\alpha}(\overline{\mathrm{T}_{1}})}\right).
$$
\end{proposition}

\begin{proof}
We prove the result for equations without dependence on $Du$ and $u$, thus, it follows the proposition(see \cite[Theorem 4.3]{Winter} for details). We note that we can  approximate $f$ in $L^{\Upsilon}_{\omega}$ by functions $f_j \in C^{\infty}(\overline{\mathrm{B}^+_1}) \cap L^{\Upsilon}_{\omega}(\mathrm{B}^+_1)$ such that $f_{j}\to f$ in $L^{\Upsilon}_{\omega}(\mathrm{B}^{+}_{1})$ and also approximate $g$ by a sequence $(g_{j})$ in $C^{1,\alpha}(\mathrm{T}_{1})$ with the property that $g_{j}\to g$ in $C^{1,\alpha}(\mathrm{T}_{1})$. By the existence and uniqueness theorem, Theorem \ref{Existencia}, there exists a sequence of functions $ u_{j}\in C^{0}(\overline{\mathrm{B}^{+}_{1}})$, which they are viscosity solutions of following family of PDEs:
$$
\left\{
\begin{array}{rclcl}
F(D^2u_j, x) &=& f_j(x) & \mbox{in} & \mathrm{B}^+_1,\\
\mathfrak{B}(Du_{j},u_{j},x)&=& g_{j}(x) & \mbox{on} & \mathrm{T}_{1}\\
u_{j}(x)&=&u(x) & \mbox{on} & \partial \mathrm{B}^{+}_{1}\setminus \mathrm{T}_{1}.
\end{array}
\right.
$$
Therefore, the assumptions of the Proposition $\ref{solution4entradas}$   are in force. By these result, we have that
$$
\|u_j\|_{W^{2,\Upsilon}_{\omega}\left(\mathrm{B}^{+}_{\frac{1}{2}}\right)} \le \mathrm{C}.\left( \|u_j\|^{n}_{L^{\infty}(\mathrm{B}^+_1)} + \|f_j\|_{L^{\Upsilon}_{\omega}(\mathrm{B}^+_1)}+\Vert g_{j}\Vert_{C^{1,\alpha}(\overline{\mathrm{T}_{1}})}\right),
$$
for a universal constant $\mathrm{C}>0$. Furthermore, a standard covering argument yields $u_j \in W^{2,\Upsilon}_{\omega}(\mathrm{B}^+_1)$. From Lemma $\ref{ABP-fullversion}$, $(u_j)_{j \in \mathbb{N}}$ is uniformly bounded in $W^{2,\Upsilon}_{\omega}(\overline{\mathrm{B}^+_{\rho}})$ for $\rho  \in (0, 1)$.

Once again, we can by ABP Maximum Principle \ref{ABP-fullversion} and Lemma \ref{mergulhoorliczlebesgue} we obtain,
\begin{eqnarray*}
\|u_j-u_k\|_{L^{\infty}\left(\mathrm{B}^+_{1}\right)} &\le& \mathrm{C}(n,\lambda,\Lambda, \mu_0)(\|f_j-f_k\|_{L^{n}(\mathrm{B}^+_1)}+\Vert g_{j}-g_{k}\Vert_{C^{1,\alpha}(\mathrm{T}_{1})})\\
&\leq&C(n,\lambda,\Lambda,\mu_{0},\omega,p_{0},p_{2},i(\Phi) )(\|f_{j}-f_{k}\|_{L^{\Upsilon}_{\omega}(\mathrm{B}^{+}_{1})}+\Vert g_{j}-g_{k}\Vert_{C^{1,\alpha}(\mathrm{T}_{1})}).
\end{eqnarray*}
Thus, $u_j \to u_{\infty} \quad \textrm{in} \quad C^0(\overline{\mathrm{B}^+_1})$. Moreover, since $(u_j)_{j \in \mathbb{N}}$ is bounded in $W^{2,\Upsilon}_{\omega}\left(\mathrm{B}^+_{\frac{1}{2}}\right)$. Due to the fact that $\Upsilon\in\Delta_{2}\cap\nabla_{2}$ we have that $W^{2,\Upsilon}_{\omega}(\mathrm{B}^{+}_{\frac{1}{2}})$ is a reflexive Banach space, it follows that we can assume unless passing by the subsequence that  $u_j \rightharpoonup u_{\infty}$ weakly in $W^{2,\Upsilon}_{\omega}\left(\mathrm{B}^+_{\frac{1}{2}}\right)$.
Thus,
$$
\|u_{\infty}\|_{W^{2,\Upsilon}_{\omega}\left(\mathrm{B}^{+}_{\frac{1}{2}}\right)} \le \mathrm{C} \cdot\left( \|u_{\infty}\|^{n}_{L^{\infty}(\mathrm{B}^+_1)} + \|f\|_{L^{\Upsilon}_{\omega}(\mathrm{B}^+_1)}+\Vert g\Vert_{C^{1,\alpha}(\overline{\mathrm{T}_{1}})}\right).
$$
On the other hand, by embed Lemma \ref{mergulhoorliczlebesgue} it follows that $ W^{2,\Upsilon}_{\omega}(\mathrm{B}^{+}_{1})$ is continuously embedded in $W^{2,p}(\mathrm{B}^{+}_{1})$ and consequently be the Sobolev's embedded $W^{2,\Upsilon}_{\omega}(\mathrm{B}^{+}_{1})$ is continuously embedded in $C^{1,1-\frac{n}{p}}(\overline{\mathrm{B}^{+}_{1}})$. Hence, we are in the hypotheses of \cite[Corollary 3.8]{Bessa} and ensure  punctually that $u_{\infty}$ satisfies $\mathfrak{B}(Du_{\infty},u_{\infty},x)=g(x)$ in $\mathrm{T}_{1}$. Finally, stability results (see Lemma \ref{Est}) ensure that $u_{\infty}$ is an $L^{p}-$viscosity solution to
$$
\left\{
\begin{array}{rclcl}
F(D^2u_{\infty}, x) &=& f(x)& \mbox{in} & \mathrm{B}^+_1,\\
\mathfrak{B}(Du_{\infty},u_{\infty},x)&=& g(x) & \mbox{on} & \mathrm{T}_1 \\
u_{\infty}(x)&=& u(x) & \mbox{on} & \partial \mathrm{B}^{+}_{1}\setminus \mathrm{T}_{1}.
\end{array}
\right.
$$
Thus, $w \defeq u_{\infty}-u$ fulfills
$$
\left\{
\begin{array}{rclcl}
w\in S(\frac{\lambda}{n}, \Lambda,0) & \mbox{in} & \mathrm{B}^+_1,\\
\mathfrak{B}(Dw,w,x)= 0 & \mbox{on} & \mathrm{T}_1 \\
w= 0 & \mbox{on} & \partial \mathrm{B}^{+}_{1}\setminus \mathrm{T}_{1}.
\end{array}
\right.
$$
which by Lemma $\ref{ABP-fullversion}$ we conclude that $w=0$ in $\overline{\mathrm{B}^{+}_{1}}\setminus \mathrm{T}_{1}$. By continuity, $w=0$ in $\overline{\mathrm{B}^{+}_{1}}$ which finishes the proof.
\end{proof}

\bigskip

With the Proposition \ref{Casolp} we can give a proof for Theorem \ref{T1}.

\begin{proof}[{\bf Proof of Theorem \ref{T1}}]
Based on standard reasoning (cf.\cite{Winter} and  \cite{ZhangZhengmorrey}) it is important to highlight that is always possible to perform a change of variables in order to flatten the boundary, since $\partial \Omega \in C^{2,\alpha}$. For this end, consider $x_0 \in \partial \Omega$, by $\partial \Omega \in C^{2,\alpha}$ there exists a neighborhood of $x_0$, which we will label $\mathcal{V}(x_0)$ and a $C^{2, \alpha}-$diffeomorfism $\Phi: \mathcal{V}(x_0) \to \mathrm{B}_1(0)$ such that $    \Phi(x_0) = 0 \quad \mbox{and} \quad \Phi(\Omega \cap \mathcal{V}(x_0)) = \mathrm{B}^{+}_1$. Now, we define $\tilde{u} \defeq u \circ \Phi^{-1} \in C^0(\mathrm{B}^+_1)$. Next, observe that $\tilde{u}$ is an $L^{p}-$viscosity solution to
$$
\left\{
\begin{array}{rclcl}
 \tilde{F}(D^2 \tilde{u}, D \tilde{u}, \tilde{u}, y) &=& \tilde{f}(x) & \mbox{in} & \mathrm{B}^+_1,\\
\tilde{\mathfrak{B}}(D\tilde{u},\tilde{u},x) & = & \tilde{g}(x) &\mbox{on} & \mathrm{T}_1.
\end{array}
\right.
$$
where for $y=\Phi^{-1}(x)$ we have
$$
\left\{
\begin{array}{rcl}
\tilde{F}(\mathrm{X}, \varsigma, \eta, y) &=& F\left(D \Phi^t(y) \cdot \mathrm{X} \cdot D \Phi(y) + \varsigma D^2\Phi, \varsigma D\Phi(y), \eta, y\right)\\
\tilde{f}(x) & \defeq & f \circ \Phi^{-1}(x)\\
\tilde{\mathcal{B}}(\overrightarrow{v},s,x)&\defeq& \tilde{\beta}(y)\cdot \overrightarrow{v}+\tilde{\gamma}(x)s\\
\tilde{\beta}(x) & \defeq & (\beta \circ \Phi^{-1}) \cdot (D \Phi \circ \Phi^{-1})^{t}\\
\tilde{\gamma}(x) & \defeq & (\gamma \circ \Phi^{-1}) \cdot (D \Phi \circ \Phi^{-1})^{t}\\
\tilde{g}(x) & \defeq & g \circ \Phi^{-1}\\
\tilde{\omega}(x)&=&\omega\circ \Phi^{-1}(x).
\end{array}
\right.
$$
Furthermore, note that $\tilde{F}$ is a uniformly elliptic operator with ellipticity constants $\lambda \mathrm{C}(\Phi)$, $\Lambda \mathrm{C}(\Phi)$ and $\tilde{\omega}\in \mathcal{A}_{i(\Phi)}$ because of change-of-variables theorem for the Lebesgue integral.  
Thus,
$$
\tilde{F}^{\sharp}(\mathrm{X}, \varsigma, \eta, x) = F^{\sharp}\left(D\Phi^t(\Phi^{-1}(x)) \cdot \mathrm{X}\cdot D\Phi(\Phi^{-1}(x)) +  \varsigma D^2 \Phi(\Phi^{-1}(x)),\varsigma D\Phi(y), \eta, \Phi^{-1}(x)\right).
$$
In consequence, we conclude that $\psi_{\tilde{F}^{\sharp}}(x,x_0) \le \mathrm{C}(\Phi) \psi_{F^{\sharp}}(x,x_0)$, which ensures that $\tilde{F}$ falls into the assumptions of Proposition \ref{Casolp}. Thus, $u\in W^{2,\Upsilon}_{\omega}(\mathrm{B}^{+}_{\frac{1}{2}})$ and 
\begin{eqnarray*}\label{estimativa3.6}
\|u\|_{W^{2,\Upsilon}_{\omega}(\mathrm{B}^{+}_{\frac{1}{2}})}\leq C(\|u\|^{n}_{L^{\infty}(\Omega)}+\|f\|_{L^{\Upsilon}_{\omega}(\Omega)}+\|g\|_{C^{1,\alpha}(\partial\Omega)}).
\end{eqnarray*}
Coupled with the inner guesses given in \cite[Theorem 6.2]{Lee} , the proof of the Theorem follows by standard covering argument.
\end{proof}

\begin{remark}
We can apply the Theorem \ref{T1}  to obtain the density of $W^{2,\Upsilon}_{\omega}$ in the  of class $C^{0}$-viscosity solutions. More precisely, following the proof of \cite[Theorem 6.1]{Bessa} with minimal alterations we guarantee which if $u$ be a $C^{0}-$viscosity solution of
$$
\left\{
\begin{array}{rclcl}
	F(D^{2}u,x) & = & f(x) & \text{in} & \mathrm{B}^{+}_{1} \\
	\mathfrak{B}(Du,u,x) & = & g(x) & \text{on} & \mathrm{T}_{1},
\end{array}
\right.
$$
where $f\in L^{\Upsilon}_{\omega}(\mathrm{B}^{+}_{1})\cap C^{0}(\mathrm{B}^{+}_{1})$ ($\Upsilon$ as in Proposition \ref{Casolp}), $\omega\in A_{i(\Phi)}$, $\beta,\gamma, g\in C^{1,\alpha}(\overline{\mathrm{T}_{1}})$ with $\gamma\leq 0$ and $\beta \cdot \nu \geq \mu_{0}$ on $\mathrm{T}_{1}$, for some $\mu_{0}>0$. Then, for any $\varepsilon>0$, there exists a sequence $(u_{j})_{j\in\mathbb{N}}\subset (W^{2,\Upsilon}_{\omega})_{loc}(\mathrm{B}^{+}_{1})\cap \mathcal{S}(\lambda-\varepsilon,\Lambda+\varepsilon,f)$ converging local uniformly to $u$.
\end{remark}
Now we will discuss another version of Theorem \ref{T1}. Note that on the right side of our estimate of this theorem the term $\|u\|_{L^{\infty}(\Omega)}$ appears raised to the power of degree $n$. We can improve this estimate without dependence on this term, however we need more regularity in the terms that make up the $\mathfrak{B}$ operator that governs the oblique boundary condition.
\begin{theorem}\label{Teoremasemotermou}
Let  $\Omega\subset \mathbb{R}^{n}$($n\geq 2$) be a bounded domain with $\partial \Omega \in C^{2,\alpha}$ $(\alpha\in(0,1))$. Assume that structural assumptions (H1)-(H5) are in force with the proviso that $\beta,\gamma\in C^{2}(\partial\Omega)$ instead of belonging to $C^{1,\alpha}(\partial\Omega)$ and $u$ be an $L^{p}-$viscosity solution of \eqref{1.1} where $p=p_{0}n$ for the constant $p_{0}>1$ of Lemma \ref{mergulhoorliczlebesgue}. Then, $u \in W^{2,\Upsilon}_{\omega}(\Omega)$ where $\Upsilon(t)=\Phi(t^{n})$, with the estimate 
\begin{equation} \label{estimativasemnormadeu}
\|u\|_{W^{2,\Upsilon}_{\omega}(\Omega)} \le \mathrm{C}\cdot\left(\|f\|_{L^{\Upsilon}_{\omega}(\Omega)}+\| g\|_{C^{1,\alpha}(\partial \Omega)}   \right),
\end{equation}
where $\mathrm{C}$ is a positive constant that depends only on $n$,  $\lambda$,  $\Lambda$, $\xi$, $\sigma$, $p_{0}$, $p_{2}$, $\Phi$, $\omega$, $\mu_{0}$, $\|\beta\|_{C^{1,\alpha}(\partial\Omega)}$, $\|\gamma\|_{C^{1,\alpha}(\partial\Omega)}$ and $\|\partial \Omega\|_{C^{2,\alpha}}$.
\end{theorem}
\begin{proof}
Indeed, by assumptions we can use Theorem \ref{T1} to conclude that $u\in W^{2,\Upsilon}_{\omega}(\Omega)$.
To conclude the proof of the theorem, it remains for us to show that the estimate \eqref{estimativasemnormadeu} is valid. Indeed, suppose by absurdity that the estimate \ref{estimativasemnormadeu} is false. Then, there exists sequences $(F_{k})_{k\in\mathbb{N}}$, $(u_{k})_{k\in\mathbb{N}}$, $(f_{k})_{k\in\mathbb{N}}$ and $(g_{k})_{k\in\mathbb{N}}$ such that, $u_{k}$ is $L^{p}$-viscosity solution of 
$$
\left\{
\begin{array}{rclcl}
F_{k}(D^2u_{k},Du_{k},u_{k},x) &=& f_{k}(x)& \mbox{in} &   \Omega \\
\beta \cdot Du_{k} + \gamma u_{k}&=& g_{k} &\mbox{on}& \partial \Omega,
\end{array}
\right.
$$
where $u_{k}\in W^{2,\Upsilon}_{\omega}(\Omega)$, $f_{k}\in L^{\Upsilon}_{\omega}(\Omega)$ and $g_{k}\in C^{2}(\partial \Omega)$ and 
\begin{eqnarray}\label{estimativa1doteorema3.5}
\|u_{k}\|_{W^{2,\Upsilon}_{\omega}(\Omega)}>k(\|f_{k}\|_{L^{\Upsilon}_{\omega}(\Omega)}+\|g_{k}\|_{C^{1,\alpha}(\partial \Omega)}), \forall k\in\mathbb{N}.
\end{eqnarray}
By normalization argument, we can suppose that $\|u_{k}\|_{L^{\Upsilon}_{\omega}(\Omega)}=1$ for all $k\in\mathbb{N}$. By \eqref{estimativa1doteorema3.5}, it follows that $f_{k}\to 0$ in $L^{\Upsilon}_{\omega}(\Omega)$ and $g_{k}\to 0$ in $C^{1,\alpha}(\partial\Omega)$ when $k\to \infty$. Futhermore, unless passing the subsequence, we have $u_{k}\rightharpoonup u_{0}$ in $W^{2,\Upsilon}_{\omega}(\Omega)$. By Lemma \ref{mergulhoorliczlebesgue} and Sobolev's embedded we have that $u_{k}\to u_{0}$ in $C^{1,1-\frac{n}{p}}(\overline{\Omega})$ when $k\to \infty$. Moreover, by structural condition $(H1)$ $F_{k}$ converges locally uniformly to an $(\lambda,\Lambda,\sigma,\xi)$-elliptic operator $F_{\infty}$. Thus, by Stability Lemma \ref{Est}, $u_{0}$ satisfies in viscosity sense
\begin{eqnarray}\label{solucaounica}
\left\{
\begin{array}{rclcl}
F_{\infty}(D^2u_{0},Du_{0},u_{0},x) &=& 0& \mbox{in} &   \Omega \\
\beta \cdot Du_{0} + \gamma u_{0}&=& 0 &\mbox{on}& \partial \Omega.
\end{array}
\right.
\end{eqnarray}
Now since the operator $\mathfrak{B}(p,s,x)=\beta(x)\cdot p+\gamma(x)s$ is of class $C^{2}$ in $p$ and $x $ since $\beta,\gamma\in C^{2}(\partial\Omega)$ it follows from \cite[Theorem 7.19]{Lieberman} that the problem \eqref{solucaounica} admits a unique solution and since $v=0$ is clearly a solution it follows that $u_{0}=0$. But for each $k\in\mathbb{N}$ the estimate is valid according to Theorem \ref{T1},
\begin{eqnarray*}
\|u_{k}\|_{W^{2,\Upsilon}_{\omega}(\Omega)}\leq \mathrm{C}(\|u_{k}\|^{n}_{L^{\infty}(\Omega)}+\|f_{k}\|_{L^{\Upsilon}_{\omega}(\Omega)}+\|g_{k}\|_{C^{1,\alpha}(\partial \Omega)}).
\end{eqnarray*}
But by weak convergence $u_{k}\rightharpoonup u_{0}$, we get
\begin{eqnarray*}
\|u_{0}\|_{W^{2,\Upsilon}_{\omega}(\Omega)}\leq\liminf_{k \to +\infty}\|u_{k}\|_{W^{2,\Upsilon}_{\omega}(\Omega)}\leq \mathrm{C}\liminf_{k \to +\infty}(\|u_{k}\|^{n}_{L^{\infty}(\Omega)}+\|f_{k}\|_{L^{\Upsilon}_{\omega}(\Omega)}+\|g_{k}\|_{C^{1,\alpha}(\partial\Omega)})=0,
\end{eqnarray*}
which is a contradiction, since $\|u_{k}\|_{W^{2,\Upsilon}_{\omega}(\Omega)}=1$ for all $k\in\mathbb{N}$. This ends the proof.
\end{proof}

\section{Weighted Orlicz estimates for obstacle problem}\label{obstaculo}

We will present in this part the proof for the obstacle problem of Theorem \ref{T2}.

\begin{proof}[\bf {Proof of Theorem \ref{T2}}]
	Consider $\varepsilon>0$ arbitrarily fixed and choosen $\Xi_{\varepsilon}$ smooth function in $\mathbb{R}$ such that
	$$
	\Xi_{\varepsilon}(s) \equiv 0 \quad \textrm{if} \quad s \le 0; \quad \Xi_{\varepsilon}(s) \equiv 1 \quad \textrm{if} \quad s \ge \varepsilon,
	$$
	$$
	0 \le \Xi_{\varepsilon}(s) \le 1 \quad \textrm{for any} \quad s \in \mathbb{R}.
	$$
	and
\begin{eqnarray*}
h(x):=f(x)-F(D^{2}\psi,D\psi,\psi,x).
\end{eqnarray*}
By the $(H1)$ condition and the hypotheses we have a.e. $x\in \Omega$ that
\begin{eqnarray*}
|h(x)|\leq |f(x)|+\mathrm{C}(n,\lambda,\Lambda,\sigma,\xi)(|\psi|+|D\psi|+|D^{2}\psi|)
\end{eqnarray*}
and so, by $\psi\in W^{2,\Upsilon}_{\omega}(\Omega)$ and $f\in L^{\Upsilon}_{\omega}(\Omega)$, it follows that $h\in L^{\Upsilon}_{\omega}(\Omega)$ with the estimate
\begin{eqnarray}\label{5.2}
\|h\|_{L^{\Upsilon}_{\omega}(\Omega)}\leq C\cdot (\|f\|_{L^{\Upsilon}_{\omega}(\Omega)}+\|\psi\|_{W^{2,\Upsilon}_{\omega}(\Omega)})
\end{eqnarray} 
where $C=C(,n,\lambda,\Lambda,\sigma,\xi)>0$. Now we consider the following penalized problem
	\begin{equation} \label{5.3}
		\left\{
		\begin{array}{rclcl}
			F(D^2u_{\varepsilon},Du_{\varepsilon},u_{\varepsilon},x) &=& \mathrm{h}^+(x) \Xi_{\varepsilon}(u_{\varepsilon} - \psi) + f(x) - \mathrm{h}^+(x)& \mbox{in} & \Omega \\
			\mathfrak{B}(Du_{\varepsilon},u_{\varepsilon},x)&=& g(x)  & \mbox{on} &\partial \Omega,
		\end{array}
		\right.
	\end{equation}
Our goal now is to guarantee the existence and uniqueness of a viscosity solution for \eqref{5.3}. For this, for each $v_{0}\in L^{\Upsilon}_{\omega}(\Omega)$ fixed, by Perron's Methods(see Lieberman's book \cite[Theorem 7.19]{Lieberman}) and Theorem \ref{Teoremasemotermou} there exists a unique viscosity solution $u_{\varepsilon}\in W^{2,\Upsilon}_{\omega}(\Omega)$ for the problem 
\begin{eqnarray*}
\left\{
\begin{array}{rclcl}
F(D^2u_{\varepsilon},Du_{\varepsilon},u_{\varepsilon},x) &=& \mathrm{h}^+(x) \Xi_{\varepsilon}(v_{0} - \psi) + f(x) - \mathrm{h}^+(x)& \mbox{in} & \Omega \\
\mathfrak{B}(Du_{\varepsilon},u_{\varepsilon},x)&=& g(x)  & \mbox{on} &\partial \Omega,
\end{array}
\right.
\end{eqnarray*}
with estimate,
\begin{eqnarray}\label{5.4}
\|u_{\varepsilon}\|_{W^{2,\Upsilon}_{\omega}(\Omega)}\leq C\cdot (\|\hat{f}_{v_{0}}\|_{L^{\Upsilon}_{\omega}(\Omega)}+\|g\|_{C^{1,\alpha}(\partial \Omega)}),
\end{eqnarray}
where $C=C(n,p_{0},p_{2},\lambda,\Lambda,\xi,\sigma,\mu_{0},\omega,i(\Phi), \|\beta\|_{C^{1,\alpha}(\partial\Omega)},\|\gamma\|_{C^{1,\alpha}(\partial \Omega)}, \theta_0,\mbox{diam}(\Omega))>0$  for the function
$$\hat{f}_{v_{0}}(x)=\mathrm{h}^+(x) \Xi_{\varepsilon}(v_{0} - \psi) + f(x) - \mathrm{h}^+(x).$$
But, by definitions of $\hat{f}_{v_{0}}$ and $\Xi_{\varepsilon}$  and \eqref{5.2}, we have the following estimate for $L^{\Upsilon}_{\omega}$ norm of $\hat{f}_{v_{0}}$
\begin{eqnarray*}
\|\hat{f}_{v_{0}}\|_{L^{\Upsilon}_{\omega}(\Omega)}\leq C(n,\lambda,\Lambda,\sigma,\xi)\cdot (\|f\|_{L^{\Upsilon}_{\omega}(\Omega)}+\|\psi\|_{W^{2,\Upsilon}_{\omega}(\Omega)}).
\end{eqnarray*}
Thus, we have at \eqref{5.4} that
\begin{eqnarray}\label{5.5}
\|u_{\varepsilon}\|_{W^{2,\Upsilon}_{\omega}(\Omega)}\leq C\cdot (\|f\|_{L^{\Upsilon}_{\omega}(\Omega)}+\|\psi\|_{W^{2,\Upsilon}_{\omega}(\Omega)}+\|g\|_{C^{1,\alpha}(\partial \Omega)})\defeq \tilde{C}_{0},
	\end{eqnarray}
	where $C=C(n,p_{0},p_{2},\lambda,\Lambda,\sigma,\xi,\mu_{0},\omega,i(\Phi),\|\beta\|_{C^{2}(\partial\Omega)},\|\gamma\|_{C^{1,\alpha}(\partial \Omega)},\theta_0,\mbox{diam}(\Omega))>0$  is independent of $v_{0}$ and $\varepsilon$. This ensures that the  operator $T:L^{\Upsilon}_{\omega}(\Omega)\to W^{2,\Upsilon}_{\omega}(\Omega)\subset L^{\Upsilon}(\Omega)$ given by $T(v_{0})=u_{\epsilon}$ applies closed balls over $\tilde{C}_{0}$-balls. Thus, $T$ is compact operator and consequently we can use the Schauder's fixed point theorem (\cite[pp. 179]{Brezis}), there exists $u_{\varepsilon}$ such that $T(u_{\varepsilon})=u_{\varepsilon}$, that is, $u_{\varepsilon}$ is viscosity solution of $\eqref{5.3}$.
	
By the above construction we guarantee that the sequence of solutions $\{u_{\varepsilon}\}_{\varepsilon>0}$ is universally bounded in $W^{2,\Upsilon}_{\omega}(\Omega )$ . Now by Lemma \ref{mergulhoorliczlebesgue} and Sobolev's embedded and the compactness arguments, we can find subsequence $\{u_{\varepsilon_{k}}\}_{k\in\mathbb{N}}$ with $\varepsilon_{k}\to 0$ and a function $u_{\infty}\in W^{2,\Upsilon}_{\omega}(\Omega)$ such that $u_{\varepsilon_{k}}\rightharpoonup u_{\infty}$ in $W^{2,\Upsilon}_{\omega}(\Omega)$, $u_{\varepsilon_{k}}\to u_{\infty}$ in $C^{0,\tilde{\alpha}_{0}}$ and $C^{1,\hat{\alpha}}(\overline{\Omega})$ for some $\tilde{\alpha}_{0},\hat{\alpha}=\hat{\alpha}(n,p)\in (0,1)$. 
	
Under these conditions we claim that $u_{\infty}$ is a solution of viscosity of \eqref{obss1}. In fact, by  construction and Arzelá-Ascoli theorem we have that the condition $$\beta(x)\cdot Du_{\infty}(x)+\gamma(x)u_{\infty}(x)=g(x)$$
is satifies in viscosity sense(more precisely, point-wisely). On other hand, by Stability Lemma \ref{Est} it follows that, a fact
\begin{eqnarray*}
F(D^2 u_{\varepsilon_k}, Du_{\varepsilon_k}, u_{\varepsilon_k},x) = \mathrm{h}^+(x) \Xi_{\varepsilon_k}(u_{\varepsilon_k} - \psi) + f(x) - \mathrm{h}^+(x)\le f(x) \quad \textrm{in} \quad \Omega \quad  \forall k \in \mathbb{N}
\end{eqnarray*}
implies that,
\begin{eqnarray*}
F(D^2 u_{\infty}, Du_{\infty}, u_{\infty}, x) \le f \quad \textrm{in} \quad \Omega.
\end{eqnarray*}
Now, we prove that
$$
u_{\infty} \ge \psi \quad \textrm{in} \quad \overline{\Omega}.
$$
In effect, rescalling that $\Xi_{\varepsilon_k}(u_{\varepsilon_k} - \psi) \equiv 0$ on the set $\mathcal{O}_k \defeq \{x \in \overline{\Omega} \, : \, u_{\varepsilon_k}(x) < \psi(x)\}$.  Observe that, if $\mathcal{O}_k = \emptyset$, it follows the claim. Thus, we suppose that $\mathcal{O}_k \not= \emptyset$. Then,
$$
F(D^2 u_{\varepsilon_k}, Du_{\varepsilon_k}, u_{\varepsilon_{k}},x) = f(x) - \mathrm{h}^+(x) \quad \textrm{for} \,\,\, x \in \mathcal{O}_k.
$$
Now, note that $\mathcal{O}_k$ is relatively open in $\overline{\Omega}$ for each $k \in \mathbb{N}$ since $u_{\varepsilon_k} \in C(\overline{\Omega})$. Moreover, from definition of $\mathrm{h}$ we have
$$
F(D^2 \psi, D \psi, \psi, x ) = f(x)-\mathrm{h}(x) \ge F(D^2 u_{\varepsilon_k}, Du_{\varepsilon_k}, u_{\varepsilon_k}, x) \,\,\, \textrm{in} \,\,\, \mathcal{O}_k
$$
Moreover, we have $u_{\varepsilon_k} = \psi$ on $\partial \mathcal{O}_k \setminus \partial \Omega$. Thus, from the Comparison Principle (see \cite[Theorem 2.10]{CCKS} and \cite[Theorem 7.17]{Lieberman}) we conclude that $u_{\varepsilon_k} \ge \psi$ in $\mathcal{O}_k$,  which is a contradiction. Therefore, $\mathcal{O}_{k} = \emptyset$ and $u_{\infty} \ge \psi$ in $\overline{\Omega}$.
	
In order to complete the proof of the claim, we must show that
$$
F(D^2 u_{\infty}, Du_{\infty}, u_{\infty} ,x) = f(x) \,\,\,\, \textrm{in} \,\,\,\, \{ u_{\infty} > \psi \}
$$
in the viscosity sense. But this fact follows directly from the following observation,
$$
\mathrm{h}^+(x)\Xi_{\varepsilon_k}(u_{\varepsilon_k} -\psi) + f(x)-\mathrm{h}^{+}(x) \to f(x) \,\,\, \textrm{a.e. in} \,\,\, \left\{x \in \Omega \, : \, u_{\infty}(x) > \psi(x) + \frac{1}{j} \right\}
$$
and consequently by Stability Lemma \ref{Est} we have that in the viscosity sense
$$
F(D^2 u_{\infty}, Du_{\infty}, u_{\infty} ,x) = f(x) \,\,\, \textrm{in} \,\,\, \{u_{\infty} > \psi\} = \bigcup_{j=1}^{\infty} \left\{ u_{\infty} > \psi + \frac{1}{j}\right\} \quad \text{as} \quad  j \to +\infty,
$$
thereby proving the claim.

Thus, from \eqref{5.5} we conclude that $u_{\infty}$ satisfies
\begin{eqnarray*}
\|u_{\infty}\|_{W^{2,\Upsilon}_{\omega}(\Omega)} \leq \liminf_{k \to +\infty} \|u_{\varepsilon_k}\|_{W^{2,\Upsilon}_{\omega}(\Omega)} \leq \mathrm{C}\cdot\left( \|f\|_{L^{\Upsilon}_{\omega}(\Omega)} + \| \psi \|_{W^{2,\Upsilon}_{\omega}(\Omega)}+ \|g\|_{C^{1,\alpha}(\partial \Omega)}\right)
\end{eqnarray*}
Finally, analogously to \cite[Corollary 5.1]{Bessa} we have the uniqueness of the solution of the problem \eqref{obss1} and because $u$ is also a solution of this problem by assumption that $u=u_{\infty}$. Which ends the proof.
\end{proof}

Now notice that we can have a version of the theorem above with less regularity in terms $\beta$ and $\gamma$ starting from the use of Theorem \ref{T1} instead of Theorem \ref{Teoremasemotermou}. In that case, we pay the price for that lower regularity with a significant change in estimate \eqref{obs2}. This is the content of the following

\begin{theorem}
Let $u$ be an $L^{p}$-viscosity solution of \eqref{obss1}, where $p=p_{0}n$ , $F$ satisfies the structural assumption (H1)-(H5) and (O1)-(O2), $\partial \Omega \in C^{3}$ and $\psi \in W^{2,\Upsilon}_{\omega}(\Omega)$, where $\Upsilon(t)=\Phi(t^{n})$ and $\omega\in \mathcal{A}_{i(\Phi)}$. Then, $u \in W^{2,\Upsilon}_{\omega}(\Omega)$  and 
\begin{equation*} \label{obs3}
\|u\|_{W^{2,\Upsilon}_{\omega}(\Omega)} \le \mathrm{C}\cdot \left( \|f\|_{L^{\Upsilon}_{\omega}(\Omega)}+ \|\psi\|_{W^{2,\Upsilon}_{\omega}(\Omega)}  +\max\{\|g\|_{C^{1,\alpha}(\partial \Omega)},\|g\|^{n}_{C^{1,\alpha}(\partial \Omega)}\} \right).
\end{equation*}
where $\mathrm{C}=\mathrm{C}(n,\lambda,\Lambda, p,p_{2},\mu_0, \sigma, \xi,\omega ,i(\Phi), \|\beta\|_{C^2(\partial \Omega)},\|\gamma\|_{C^{1,\alpha}(\partial \Omega)}, \partial \Omega, \textrm{diam}(\Omega), \theta_0)$.
\end{theorem}
\begin{proof}
Similar to the proof of Theorem \ref{T2}, with the following adjustments: when we take the family  $\{u_{\varepsilon}\}_{\varepsilon>0}$ of solutions for \eqref{5.3}, we guarantee that they are in $W^{2,\Upsilon }_{\omega}(\Omega)$ via Theorem \ref{T1} instead of Theorem \ref{Teoremasemotermou} and thus we have the following estimates,
\begin{eqnarray*}
\|u_{\varepsilon}\|_{W^{2,\Upsilon}_{\omega}(\Omega)}\leq C\cdot (\|u_{\varepsilon}\|^{n}_{L^{\infty}(\Omega)}\|f\|_{L^{\Upsilon}_{\omega}(\Omega)}+\|\psi\|_{W^{2,\Upsilon}_{\omega}(\Omega)}+\|g\|_{C^{1,\alpha}(\partial \Omega)}),
\end{eqnarray*}
Using the ABP Estimate \ref{ABP-fullversion} and the Lemma \ref{mergulhoorliczlebesgue} we can conclude above that
\begin{eqnarray*}
\|u_{\varepsilon}\|_{W^{2,\Upsilon}_{\omega}(\Omega)}\leq C\cdot (\|f\|_{L^{\Upsilon}_{\omega}(\Omega)}+\|\psi\|_{W^{2,\Upsilon}_{\omega}(\Omega)}+\max\{\|g\|_{C^{1,\alpha}(\partial \Omega)},\|g\|^{n}_{C^{1,\alpha}(\partial \Omega)}\}).
\end{eqnarray*}
From this follows the proof entirely analogous to the Theorem \ref{Teoremasemotermou}.
\end{proof}

\section{Weighted Orlicz-BMO type estimates}\label{section5}

Finally in this section, we will also present an application of weighted Orlicz estimates obtained in section \ref{section3} when the source term $f$ has $L^{\Upsilon}_{\omega}$-bounded mean oscillation or $L^{\Upsilon}_{\omega}-BMO$. It is known from regularity theory that solutions of \eqref{1.1} when the source term $f$ is bounded in $\Omega$ does not necessarily imply that the Hessian $D^{2}u$ is bounded.
More precisely, let us recall that a function $f\in L^{1}_{loc}(\Omega)$ is said to belong to  the $L^{\Phi}_{\omega}-BMO(\Omega)$ space for $\Phi\in \Delta_{2}\cap \nabla_{2}$ be an $N$-function and weighted $\omega$ if
\begin{eqnarray*}
	\|f\|_{L^{\Phi}_{\omega}-BMO(\Omega)}:=\sup_{\mathrm{B}\subset \Omega}\frac{\|(f-f_{\mathrm{B}})\chi_{\mathrm{B}}\|_{L^{\Phi}_{\omega}(\Omega)}}{\|\chi_{\mathrm{B}}\|_{L^{\Phi}_{\omega}(\Omega)}}<+\infty
\end{eqnarray*}
where the supremum is taken over all balls $\mathrm{B}\subset \Omega$ and for each of these balls we denote
\begin{eqnarray*}
	f_{\mathrm{B}}=\intav{\mathrm{B}}f(x)dx.
\end{eqnarray*}
For example if $\Phi(t)=t^{p}$ for $p>1$ we return the definition of spaces $p-BMO$.
\begin{remark}\label{observacaodeequivalencia}
Note that by the fact that $\Phi\in\Delta_{2}\cap\nabla_{2}$ and assuming $\omega\in \mathcal{A}_{i(\Phi)}$ follows from \cite[Theorem 2.3] {Ho} that there exist universal constants $0<\mathrm{A}\leq\mathrm{B}$ such that
\begin{eqnarray*}
\mathrm{A}\|f\|_{BMO(\mathrm{B}^{+}_{1})}\leq \|f\|_{L^{\Phi}_{\omega}-BMO(\mathrm{B}^{+}_{1})}\leq \mathrm{B}\|f\|_{BMO(\mathrm{B}^{+}_{1})}, \ \forall f\in L^{1}_{loc}(\mathrm{B}^{+}_{1})
\end{eqnarray*}
\end{remark}
Our last goal in this manuscript is to secure a priori estimates for $D^{2}u$ in the context of $L^{\Phi}_{\omega}-BMO$ spaces for the following problem
\begin{eqnarray}\label{problemalocal}
\left\{
\begin{array}{rclcl}
F(D^2u,x) &=& f(x)& \mbox{in} &   \mathrm{B}^+_1 \\
\mathfrak{B}(Du,u,x)&=& g(x) &\mbox{on}& \mathrm{T}_1.
\end{array}
\right.
\end{eqnarray}
For this purpose, we will need to assume the following extra hypothesis:
\begin{enumerate}
\item[(H4)$^{\sharp}$] (\textbf{Interior estimates $C^{2,\rho}$}) The recession operator $F^{\sharp}$ associate to $F$ fulfills a $C^{2, \rho}$ \textit{a priori} estimate up to the boundary, for some $\rho \in (0,1)$, i.e., for any $g_0 \in C^{1, \rho}(\mathrm{T}_1)$, $\beta,\gamma \in C^{1, \rho}(\mathrm{T}_1)$, there exists $\mathrm{c}^{\ast}>0$, depending only on universal parameters, such that any viscosity solution of
$$
\left\{
\begin{array}{rclcl}
F^{\sharp}(D^2 \mathfrak{h}, x) &=& 0& \mbox{in} & \mathrm{B}^+_1,\\
\mathfrak{B}(D\mathfrak{h},\mathfrak{h},x)  &=& g_0(x) & \mbox{on} &\mathrm{T}_1 .
\end{array}
\right.
$$
belong to $C^{2, \rho}(\mathrm{B}^{+}_1) \cap C^0(\overline{\mathrm{B}^+_1})$ and fulfills
$$
\|\mathfrak{h}\|_{C^{2, \rho}\left(\overline{\mathrm{B}^+_{r}}\right)} \le\mathrm{c}^{\ast}r^{-(2+\rho)} \left(\|\mathfrak{h}\|_{L^{\infty}(\mathrm{B}^+_1)} + \|g_0\|_{C^{1,\psi}(\overline{\mathrm{T}_1})}\right)\,\,\, \forall \,\,0<r \ll 1.
$$
\end{enumerate}

\begin{theorem}[{\bf $L^{\Upsilon}_{\omega}$-BMO regularity for the hessian}]\label{BMO}
Let $u$ be an $L^{p}$-viscosity solution to problem \eqref{problemalocal} where $f \in L^{\Upsilon}_{\omega}-BMO(\mathrm{B}^+_1)\cap L^{\Upsilon}_{\omega}(\mathrm{B}^{+}_{1})$, for $p=p_{0}n$ and $\Upsilon(t)=\Phi(t^{n})$ for the $\Phi$ and $\omega\in \mathcal{A}_{i(\Phi)}$ as the hypothesis $(H2)$. Assume further that assumptions (H1)-(H3) and $(H4)^{\sharp}$ are in force. Then, $D^2 u \in L^{\Upsilon}_{\omega}-BMO\left(\mathrm{B}^+_{\frac{1}{2}}\right)$, with the following estimate holds
\begin{equation}\label{BMO2}
\|D^2 u\|_{L^{\Upsilon}_{\omega}-\textrm{BMO}\left(\mathrm{B}^{+}_{\frac{1}{2}}\right)} \le \mathrm{C}\left(\|u\|^{n}_{L^{\infty}(\mathrm{B}^+_1)} + \|f\|_{L^{\Upsilon}_{\omega}-\textrm{BMO}(\mathrm{B}^+_1)}+\|g\|_{C^{1, \alpha}(\overline{\mathrm{T}_{1}})} \right),
\end{equation}
where $\mathrm{C}>0$ depends only to $n$, $\lambda$,  $\Lambda$, $p_{0}$, $\omega$, $i(\Phi)$,  $\mathrm{c}^{\ast}$, $\mu_0$, $\|\beta\|_{C^{1, \alpha}(\overline{\mathrm{T}_1})}$ and $\|\gamma\|_{C^{1, \alpha}(\overline{\mathrm{T}_1})}$.
\end{theorem}
The proof of Theorem \ref{BMO} will be carried out in two steps. The first step will approximate normalized solutions of the problem along the path of operators $F_{\tau}$ for $\tau$ and $f$ small by quadratic polynomials in the order $r^{2}$ on the half ball $\mathrm {B}^{+}_{r}$. Then we will apply Proposition \ref{Casolp} to a sequence of normalized solutions associated with the initial problem to guarantee the desired result.The analysis that we will present here is similar to that exposed in \cite{Bessa} and \cite{PT} and for that reason we will only comment on the necessary modifications.

\begin{lemma} \label{Lema5.2}
Assume the hypothesis $(H1)-(H3)$ and $(H4)^{\sharp}$. There exist universal constants $\mathrm{C}^{\ast}>0$, $\tau_0 >0$ and $r>0$, such that if $u$ is a (normalized) viscosity solution of
$$
\left\{
\begin{array}{rclcl}
F_{\tau}(D^2 u, x) &=& f(x) & \mbox{in} & \mathrm{B}^+_1,\\
\mathfrak{B}(Du,u,x)&=& g(x) & \mbox{on} & \mathrm{T}_1,
\end{array}
\right.
$$
with $\max\left\{\tau, \,\|f\|_{L^{\Upsilon}_{\omega}-BMO(\mathrm{B}^+_1)}\right\} \le \tau_0$, there exists a quadratic polynomial $\mathfrak{P}$, with $\|\mathfrak{P}\|_{\infty} \le \mathrm{C}^{\ast}$ satisfying
$$
\sup_{\mathrm{B}^+_r} |u(x)-\mathfrak{P}(x)| \le r^2.
$$
\end{lemma}
\begin{proof}
The proof is based on compactness arguments as in Lemma 4.1 and Lemma 4.2 of \cite{Bessa} with the following observation: as $f\in L^{\Upsilon}_{\omega}(\mathrm{B}^{+}_{1}) $ by the Lemma \ref{mergulhoorliczlebesgue}, we have $\|f\|_{L^{p}(\mathrm{B}^{+}_{1})}\leq \mathrm{C}\|f\|_{L^{\Upsilon}_{\omega}(\mathrm{B}^{+}_{1})}$. Hence, analogously to \cite[Corollary 4.3]{Bessa} and obtain such desired quadratic polynomial with the desired constants.
\end{proof}
\begin{proof}[\textbf{Proof of Theorem \ref{BMO}}] 
Let $u$ be a solution of \eqref{problemalocal}. By normalization argument we can taken the $\kappa\in(0,1)$ to be determined a posteriori and define the function $w(x)=\kappa u(x)$ such that is normalized $L^{p}$-viscosity solution of 
\begin{eqnarray*}
\left\{
\begin{array}{rclcl}
F_{\tau}(D^2 w, x) &=& f(x) & \mbox{in} & \mathrm{B}^+_1,\\
\tilde{\mathfrak{B}}(Dw,w,x)&=& g(x) & \mbox{on} & \mathrm{T}_1,
\end{array}
\right.
\end{eqnarray*}
where $\tau:=\kappa$, $\tilde{f}(x)=\kappa f(x)$,  $\tilde{\mathfrak{B}}(q,s,x)=\beta(x)\cdot q+\gamma(x)s$, and $\tilde{g}(x):=\kappa g(x)$. We choose the constant $\kappa$ such that $\max\{\tau,\mathrm{A}^{-1}\mathrm{B}\|\tilde{f}\|_{L^{\Upsilon}_{\omega}(\mathrm{B}^{+}_{1})} \}\leq \tau_0$ where $\tau_0$ comes from Lemma \ref{Lema5.2} and the constant $\mathrm{A},\mathrm{B}>0$ are the universal constants in \eqref{observacaodeequivalencia}. Because $\|\cdot\|_{L^{\Upsilon}_{\omega}(\mathrm{B}^{+}_{1})}$ is positively homogeneous and by the definition of $w$, prove the estimate $L^{\Upsilon}_{\omega}-BMO$ for $w$ guarantees
the desired one for $u$ and so we will prove the desired estimate for $w$. To this end, we assert that there exists a sequence of quadratic polynomials $(P_{k})_{k\in \mathbb{N}}$ in form 
\begin{eqnarray}
P_{k}(x)=a_{k}+b_{k}\cdot x+\frac{1}{2}x^{t}M_{k}x
\end{eqnarray}
satisfying the following conditions:
\begin{itemize}
\item [(i)] $F^{\sharp}(M_{k},x)=\tilde{f}_{1}=:\tilde{f}_{\mathrm{B}^{+}_{1}}$.
\item[(ii)] $\displaystyle\sup_{\mathrm{B}^{+}_{r^{k}}}|w-P_{k}|\leq r^{2k}, \ \forall k\geq 1$, where $r\in\left(0,\frac{1}{2}\right)$ is the radius of semi-ball in Lemma \ref{Lema5.2}.
\item[(iii)] For all $k\geq 0$,  $r^{2(k-1)}|a_{k}-a_{k-1}|+r^{k-1}|b_{k}-b_{k-1}|+|M_{k}-M_{k-1}|\leq \mathrm{C}^{\ast}r^{2(k-1)}$, where $\mathrm{C}^{\ast}$ is the constant of Lemma \ref{Lema5.2}.
\end{itemize}
We prove the claim by finite induction process. In effect, we set $P_{0}=P_{-1}=\frac{1}{2}x^{t}M_{0}x$, for $M_{0}\in Sym(n)$ such that $F^{\sharp}(M_{0},x)=\tilde{f}_{1}$ ande the first step is clearly satifies (case k=0). Now, suppose that for $k=0,1,\cdots,m$ have been checked and define the following auxiliar function $w_{m}:\mathrm{B}^{+}_{1}\cup\mathrm{T}_{1}\to \mathbb{R}$ by
\begin{eqnarray*}
w_{m}(x)=\frac{(w-P_{m})(r^{m}x)}{r^{2m}}
\end{eqnarray*}
By induction hypothesis, $w_{m}$ is normalized function and satisfies in viscosity sense
to
$$
\left\{
\begin{array}{rclcl}
F_{\tau}(D^2 w_{m}+M_{m}, r^{m}x) &=&f_{m}(x)=\tilde{f}(r^{m}x) & \mbox{in} & \mathrm{B}^+_1\\
\mathfrak{B}_{m}(Dw,w,x) &=& \tilde{g}_{m}(x)  & \mbox{on} & \mathrm{T}_{1}.
\end{array}
\right.
$$
where
$$ 
\left\{
\begin{array}{rcl}
\mathfrak{B}_{m}(q,s,x)&=&\beta_{m}(x)\cdot\frac{q}{r^{m}}+\gamma_{m}(x)s \\
\beta_{m}(x) & \defeq &\beta(r^{m}x)\\
\gamma_{m}(x) & \defeq & r^{m}\gamma(r^{m}x)\\
\tilde{g}_{m}(x)  & \defeq &r^{-m}\left( \tilde{g}(r^{m}x)-\beta(r^{m}x) \cdot DP_{m}(r^{m}x)-\gamma(r^{m}x)P_{m}(r^{m}x)\right)
\end{array}
\right.	
$$
By the definition of the $f_{m}$ and \eqref{observacaodeequivalencia} it follows that
\begin{eqnarray*}
\|f_{m}\|_{L^{\Upsilon}_{\omega}-BMO(\mathrm{B}^{+}_{1})}\leq \mathrm{B}\|f_{m}\|_{BMO(\mathrm{B}^{+}_{1})}=\mathrm{B}\|\tilde{f}\|_{BMO(\mathrm{B}^{+}_{r})}\leq \mathrm{B}\mathrm{A}^{-1}\|\tilde{f}\|_{L^{\Upsilon}_{\omega}-BMO(\mathrm{B}^{+}_{1})}\leq \tau_{0}.
\end{eqnarray*}
And how does the $F_{m}(M,x)=F(M+M_{m},r^{m}x)$ operator satisfy, $F^{\sharp}_{m}(M,x)=F^{\sharp}(M+M_{m},r^{m}x)$ it follows that $F^{\sharp}_{m}(M_{k},x)=\tilde{f}_{1}$. Furthermore, $F^{\sharp}_{m}$ still satisfies the structural assumption $(H4)^{\sharp}$. Therefore, we can invoke the lemma \ref{Lema5.2} and obtain the existence of a quadratic polynomial $\tilde{P}$ of the form $\tilde{P}(x)=\tilde{a}+\tilde{b }\cdot x+\frac{1}{2}x^{t}\tilde{M}x$ such that $\|\tilde{P}\|_{\infty}\leq \mathrm{C}^{\ast}$ and
\begin{eqnarray}\label{estimativa1}
\sup_{\mathrm{B}^{+}_{r}}|w_{m}-\tilde{P}|\leq r^{2}.
\end{eqnarray} 
Choosing $a_{m+1}=a_{m}+r^{2m}\tilde{a}$, $b_{m+1}=b_{m}+r^{m}\tilde{b}$ and $M_{m+1}=M_{m}+\tilde{M}$ we can rescale \eqref{estimativa1} and obtain that
\begin{eqnarray*}
\sup_{\mathrm{B}^{+}_{r^{m+1}}}|w-P_{m+1}|\leq r^{2(m+1)}
\end{eqnarray*}
Henceforth, by the definitions of $a_{m+1}, b_{m+1}$ and $M_{m+1}$ naturally follows condition (ii) for the natural number  $k=m+1$. This proves the statement.

Finally, for any $\rho\in\left(0,\frac{1}{2}\right)$, choose $k$ such that $0<r^{m+1}<\rho\leq r^{m}$. We have the following estimate
\begin{eqnarray}\label{estimativa2}
\frac{\|(D^{2}w-M_{m})\chi_{\mathrm{B}^{+}_{\rho}}\|_{L^{\Upsilon}_{\omega}(\mathrm{B}^{+}_{1/2})}}{\|\chi_{\mathrm{B}^{+}_{\rho}}\|_{L^{\Upsilon}_{\omega}(\mathrm{B}^{+}_{1/2})}}&\leq&\mathrm{C} \frac{\|(D^{2}w-M_{m})\chi_{\mathrm{B}^{+}_{r^{m}}}\|_{L^{\Upsilon}_{\omega}(\mathrm{B}^{+}_{1/2})}}{\|\chi_{\mathrm{B}^{+}_{r^{m}}}\|_{L^{\Upsilon}_{\omega}(\mathrm{B}^{+}_{1/2})}}\nonumber\\
&=&\mathrm{C} \frac{\|D^{2}w_{m}\|_{L^{\Upsilon}_{\omega}(\mathrm{B}^{+}_{r^{m}})}}{\|\chi_{\mathrm{B}^{+}_{r^{m}}}\|_{L^{\Upsilon}_{\omega}(\mathrm{B}^{+}_{1/2})}}\nonumber\\
&\leq&\mathrm{C}\|D^{2}w_{m}\|_{L^{\Upsilon}_{\omega}(\mathrm{B}^{+}_{r^{m}})}\nonumber\\
&\leq& \mathrm{C},
\end{eqnarray}
where where in the penultimate inequality we use the Lemma \ref{mergulhoorliczlebesgue} and in the last one the Proposition \ref{Casolp}. So we can conclude of \eqref{estimativa2} it follows
\begin{eqnarray*}
\frac{\|(D^{2}w-(D^{2}w)_{\mathrm{B}^{+}_{\rho}})\chi_{\mathrm{B}^{+}_{\rho}}\|_{L^{\Upsilon}_{\omega}(\mathrm{B}^{+}_{1/2})}}{\|\chi_{\mathrm{B}^{+}_{\rho}}\|_{L^{\Upsilon}_{\omega}(\mathrm{B}^{+}_{1/2})}}&\leq& \frac{\|(D^{2}w-M_{m})\chi_{\mathrm{B}^{+}_{\rho}}\|_{L^{\Upsilon}_{\omega}(\mathrm{B}^{+}_{1/2})}}{\|\chi_{\mathrm{B}^{+}_{\rho}}\|_{L^{\Upsilon}_{\omega}(\mathrm{B}^{+}_{1/2})}}+\\
&+&\frac{\|(D^{2}w)_{\mathrm{B}^{+}_{\rho}}-M_{m})\chi_{\mathrm{B}^{+}_{\rho}}\|_{L^{\Upsilon}_{\omega}(\mathrm{B}^{+}_{1/2})}}{\|\chi_{\mathrm{B}^{+}_{\rho}}\|_{L^{\Upsilon}_{\omega}(\mathrm{B}^{+}_{1/2})}}\\
&\leq&2\frac{\|(D^{2}w-M_{m})\chi_{\mathrm{B}^{+}_{\rho}}\|_{L^{\Upsilon}_{\omega}(\mathrm{B}^{+}_{1/2})}}{\|\chi_{\mathrm{B}^{+}_{\rho}}\|_{L^{\Upsilon}_{\omega}(\mathrm{B}^{+}_{1/2})}}\\
&\leq&\mathrm{C}.
\end{eqnarray*}
Thus
\begin{eqnarray*}
\|D^{2}w\|_{L^{\Upsilon}_{\omega}-BMO\left(\mathrm{B}^{+}_{\frac{1}{2}}\right)}\leq \sup_{0<\rho<\frac{1}{2}}\frac{\|(D^{2}w-(D^{2}w)_{\mathrm{B}^{+}_{\rho}})\chi_{\mathrm{B}^{+}_{\rho}}\|_{L^{\Upsilon}_{\omega}(\mathrm{B}^{+}_{1/2})}}{\|\chi_{\mathrm{B}^{+}_{\rho}}\|_{L^{\Upsilon}_{\omega}(\mathrm{B}^{+}_{1/2})}}\leq  \mathrm{C}<+\infty
\end{eqnarray*}
which ends the proof of the theorem.
\end{proof}
We end this section with the following two remarks:
\begin{itemize}
\item [(i)] The Theorem \ref{BMO} together with the Remark \ref{observacaodeequivalencia} imply that $D^{2}u\in BM	O(\mathrm{B}^{+}_{1})$ with estimate
\begin{eqnarray*}
\|D^{2}u\|_{BMO\left(\mathrm{B}^{+}_{\frac{1}{2}}\right)}\leq\mathrm{C}(\|u\|^{n}_{L^{\infty}(\mathrm{B}^{+}_{1})}+\|f\|_{L^{\Upsilon}_{\omega}(\mathrm{B}^{+}_{1})}+\|g\|_{C^{1,\alpha}(\overline{\mathrm{T}_{1}})}),
\end{eqnarray*}
for the constant $\mathrm{C}>0$ depending only on the same entities as the constant in Theorem \ref{BMO}.
\item[(ii)] Via the embedding of the BMO spaces in \cite[Lemma 1]{AZBED} we can guarantee from item (i) that $Du\in C^{0,Log-Lip}\left(\mathrm{B}^{+} _{\frac{1}{2}}\right)$ with the following estimate,
\begin{eqnarray*}
|Du(x)-Du(y)|\leq -\mathrm{C}\mathcal{D}|x-y|\log|x-y|, \ \forall x,y\in\mathrm{B}^{+}_{\frac{1}{2}}, x\neq y,
\end{eqnarray*}
where
\begin{eqnarray*}
\mathcal{D}=(\|u\|^{n}_{L^{\infty}(\mathrm{B}^{+}_{1})}+\|f\|_{L^{\Upsilon}_{\omega}(\mathrm{B}^{+}_{1})}+\|g\|_{C^{1,\alpha}(\overline{\mathrm{T}_{1}})}).
\end{eqnarray*}
and the constant $\mathrm{C}>0$ depends only on  $n$, $\lambda$,  $\Lambda$, $p_{0}$, $\omega$, $i(\Phi)$,  $\mathrm{c}^{\ast}$, $\mu_0$, $\|\beta\|_{C^{1, \alpha}(\overline{\mathrm{T}_1})}$ and $\|\gamma\|_{C^{1, \alpha}(\overline{\mathrm{T}_1})}$.
\end{itemize} 
\subsection*{Acknowledgments}

\hspace{0.4cm} J.S. Bessa was partially supported by CAPES-Brazil under Grant No. 88887.482068/2020-00. The author would like to thanks Department of Mathematics of Universidade Federal do Ceará (UFC-
Brazil) by the pleasant scientific and research atmosphere.

\end{document}